\documentclass[a4paper,11pt]{article}
\usepackage[T1]{fontenc}
\usepackage[utf8]{inputenc}
\usepackage{amsmath,amssymb,amsfonts,amsthm,mathtools}
\usepackage{mathrsfs}
\usepackage{bm}
\usepackage{graphicx}
\usepackage{subcaption}
\usepackage{float}
\usepackage{enumerate}
\usepackage{enumitem}
\usepackage{xcolor}
\usepackage{array}
\usepackage{booktabs}
\usepackage{url}
\usepackage[colorlinks=true,linkcolor=blue,citecolor=blue,urlcolor=blue]{hyperref}

\newenvironment{keywords}
{\par\vspace{.05in}\footnotesize\noindent\textbf{Key words. }\ignorespaces}
{\par\vspace{.1in}}

\newenvironment{AMS}
{\par\vspace{.05in}\footnotesize\noindent\textbf{AMS subject classifications. }\ignorespaces}
{\par\vspace{.1in}}

\numberwithin{equation}{section}
\allowdisplaybreaks[1]

\newtheorem{thm}{Theorem}[section]
\newtheorem{lem}[thm]{Lemma}
\newtheorem{cor}[thm]{Corollary}
\newtheorem{prop}[thm]{Proposition}
\newtheorem{rem}[thm]{Remark}
\newtheorem{dfn}[thm]{Definition}

\DeclareMathOperator{\supp}{supp}

\DeclareMathOperator{\Tr}{Tr}
\DeclareMathOperator*{\argmin}{arg\,min}
\DeclareMathOperator*{\argmax}{arg\,max}

\newcommand{\eps}{\varepsilon}
\newcommand{\RR}{\mathbb{R}}

\newcommand{\MM}{\mathcal{M}}
\newcommand{\GMM}{\mathrm{GMM}}
\newcommand{\Id}{\mathrm{Id}}
\newcommand{\OT}{\mathrm{OT}}
\newcommand{\POT}{\mathrm{POT}}
\newcommand{\GW}{\mathrm{GW}}
\newcommand{\MW}{\mathrm{MW}}
\newcommand{\MGW}{\mathrm{MGW}}

\newcommand{\dd}{\,\mathrm{d}}
\newcommand{\X}{\mathcal{X}}
\newcommand{\Y}{\mathcal{Y}}
\newcommand{\Z}{\mathcal{Z}}

\newcommand{\ones}{\mathbf{1}}
\newcommand{\St}{\mathrm{St}}

\title{\bf Entropic Partial Optimal Transport and Partial Gromov--Wasserstein Distance between Gaussian Mixtures}
\author{Toshiaki Yachimura \and Xiaocheng Zou}
\date{}

\begin{document}
\maketitle

\begin{abstract}
Optimal transport and Gromov--Wasserstein distances are useful tools for comparing probability measures and metric measure spaces, but their balanced formulations force all mass to be matched. This constraint is often too strong for data with outliers, missing parts, or only partial overlap. In this paper, we develop entropic partial optimal transport for Gaussian mixture models and define a partial mixture Gromov--Wasserstein distance. For the finite entropic partial optimal transport problem, we prove the existence and uniqueness of the minimizer and establish quantitative large-penalty estimates. Moreover, the resulting entropic partial component couplings induce continuous partial transport plans through Gaussian optimal maps. We analyze their large-penalty and subsequent zero-entropy limits and construct the associated displacement interpolations and barycentric projection maps. In addition, by identifying each Gaussian mixture with a finite metric measure space of Gaussian components, we establish the metric property and large-penalty limit of the partial mixture Gromov--Wasserstein distance. Finally, numerical experiments on synthetic Gaussian mixtures and point clouds illustrate the effects of the penalty and entropic regularization and the robustness of partial matching to outliers.
\end{abstract}



\pagestyle{plain}
\thispagestyle{plain}

\section{Introduction}\label{sec:introduction}

Comparing and matching probability measures is a fundamental task in modern data analysis. It appears in many areas, including machine learning, image processing, statistics, and biological data analysis, where distributions are used to represent images, point clouds, empirical samples, or high-dimensional features \cite{peyre2019computational}. A classical mathematical formulation of this problem is the Monge--Kantorovich optimal transport (OT) problem. Let $(\X,\mu)$ and $(\Y,\nu)$ be Polish spaces equipped with Borel probability measures. Let $c:\X \times \Y \to \RR_{\geq 0}$ be a measurable cost function. The Monge--Kantorovich problem is
\begin{equation}\label{eq:intro-kantorovich}
  \OT(\mu,\nu) = \inf_{\gamma \in \Pi(\mu,\nu)} \int_{\X\times\Y} c(x,y)\dd\gamma(x,y),
\end{equation}
where $\Pi(\mu,\nu)$ is the set of probability measures on $\X\times\Y$ with marginals $\mu$ and $\nu$. The theory has deep connections with geometric analysis, and it provides a geometrically meaningful way to compare probability measures \cite{ambrosio2021lectures,figalli2021invitation,santambrogio2015optimal,villani2009optimal}.

Despite its theoretical and practical importance, solving
\eqref{eq:intro-kantorovich} can be computationally expensive, especially for large-scale discrete data or high-dimensional empirical distributions. In finite-dimensional applications, a standard approach is entropic regularization, which replaces the linear program \eqref{eq:intro-kantorovich} by a smooth strictly convex problem. For probability vectors $a$ and $b$ and a cost matrix
$C = (c_{ij})$, we write
\begin{equation}\label{eq:intro-entropic-ot}
  \OT_{\eps}(a,b;C) = \min_{\omega \in \Pi(a,b)} \left\{ \sum_{i,j} c_{ij}\omega_{ij} + \eps E(\omega) \right\}, \qquad E(\omega) = \sum_{i,j}\omega_{ij}(\log\omega_{ij}-1),
\end{equation}
where $\eps > 0$ is the regularization parameter. Since the objective in \eqref{eq:intro-entropic-ot} is strictly convex, its minimizer is unique. Moreover, this minimizer can be computed efficiently by Sinkhorn-type scaling algorithms. This computational approach was introduced by Cuturi \cite{cuturi2013sinkhorn} and builds on classical matrix scaling ideas \cite{sinkhorn1964relationship,sinkhornknopp1967}. Entropic OT \eqref{eq:intro-entropic-ot} has become a standard tool in data science \cite{feydy2019interpolating,genevay2018learning,peyre2019computational}.

Although entropic regularization improves the computational efficiency of discrete OT, directly transporting high-dimensional empirical measures or continuous distributions can still be expensive. Gaussian mixture models provide a finite-dimensional representation of continuous distributions. Let
\begin{equation}\label{eq:intro-gmm}
  \mu = \sum_{k = 1}^{K} a_k\mu_k, \qquad \nu = \sum_{l = 1}^{L} b_l\nu_l
\end{equation}
be Gaussian mixtures. If $\mu_k = \mathcal N(m_{0,k},\Sigma_{0,k})$ and $\nu_l = \mathcal N(m_{1,l},\Sigma_{1,l})$, then the squared $2$-Wasserstein distance between the two Gaussian components is given by the closed formula
\begin{equation}\label{eq:intro-gaussian-w2}
  W_2^2(\mu_k,\nu_l) = \|m_{0,k}-m_{1,l}\|_2^2 + \Tr\left( \Sigma_{0,k} + \Sigma_{1,l} - 2(\Sigma_{0,k}^{1/2}\Sigma_{1,l}\Sigma_{0,k}^{1/2})^{1/2} \right).
\end{equation}
This formula is classical \cite{dowson1982frechet,gelbrich1990formula} and is the key ingredient in reducing transport between Gaussian mixtures to a finite component-level optimal transport problem.

The optimal-transport framework for Gaussian mixture models was first introduced in \cite{chen2019optimal}. This construction identifies a Gaussian mixture as a discrete measure on the space of Gaussian distributions and uses the Wasserstein distance between Gaussian components as the ground cost. Delon and Desolneux later developed this viewpoint by formulating a Wasserstein-type distance obtained by restricting continuous transport plans to Gaussian-mixture transport plans \cite{delon-2020}. Let $\GMM_{\infty}(\RR^d \times \RR^d)$ denote the set of probability measures on $\RR^d \times \RR^d$ that can be written as finite Gaussian mixtures. Delon and Desolneux define the mixture Wasserstein distance by
\begin{equation}\label{eq:intro-mw-continuous}
  \MW_2^2(\mu,\nu) = \inf_{\gamma \in \Pi(\mu,\nu) \cap \GMM_{\infty}(\RR^d \times \RR^d)} \int_{\RR^d \times \RR^d} \|x-y\|_2^2 \dd\gamma(x,y).
\end{equation}
For the Gaussian mixtures in \eqref{eq:intro-gmm}, they prove that the continuous problem \eqref{eq:intro-mw-continuous} is equivalent to the following finite component-level optimal transport problem:
\begin{equation}\label{eq:intro-mixture-wasserstein}
  \MW_2^2(\mu,\nu) = \min_{\omega \in \Pi(a,b)} \sum_{k = 1}^{K} \sum_{l = 1}^{L} W_2^2(\mu_k,\nu_l)\omega_{kl}.
\end{equation}
Thus, the restricted continuous transport problem \eqref{eq:intro-mw-continuous} reduces to a discrete OT problem between the mixture weights. Consequently, after the pairwise Gaussian costs have been computed, the remaining optimization is governed by the number of mixture components rather than by the number of sampled points or by a spatial discretization. This is the main computational advantage of the mixture formulation: the continuous OT-type problem \eqref{eq:intro-mw-continuous} is replaced by a much smaller component-level OT problem \eqref{eq:intro-mixture-wasserstein}.

The mixture Wasserstein framework has been extended in several directions. Dusson, Ehrlacher, and Nouaime developed Wasserstein-type metrics for general mixture models \cite{dusson2023wasserstein}. Wilson et al. extended Wasserstein-type distances between Gaussian mixtures to vector bundles and applied the resulting framework to statistical shape analysis \cite{wilson2024wasserstein}. Salmona, Desolneux, and Delon introduced Gromov--Wasserstein-like distances for Gaussian mixture models using pairwise Wasserstein distances between Gaussian components \cite{delon-2023}. More recently, Piening and Beinert proposed sliced variants of the mixture Wasserstein distance to reduce the computational cost of both the inner Gaussian Wasserstein computations and the outer component-level OT problem \cite{piening2025slicing}. Entropic Gaussian mixture OT has also been used in biological data analysis, including single-cell trajectory inference \cite{PMID:39710672}. These developments show that mixture-based OT provides a flexible and computationally effective way to compare structured distributions. However, these formulations remain based on balanced component-level transport. They therefore do not directly provide a mechanism for leaving unmatched, noisy, or outlying components untransported.

Partial formulations provide a natural way to address such partial matching problems in optimal transport and Gromov--Wasserstein theory. In optimal transport, partial transport allows only a prescribed amount of mass to be transported and leaves the remaining mass unmatched; it is closely related to the optimal partial transport problem and its free-boundary theory \cite{caffarelli2010free,figalli2010optimal}. Partial optimal transport and partial Gromov--Wasserstein formulations have also been studied from a computational viewpoint \cite{chapel2020partial}, and a partial Gromov--Wasserstein metric has recently been proposed for metric-measure spaces \cite{Bai2025}. 

The aim of this paper is to develop an entropic partial optimal transport (EPOT) framework for Gaussian mixtures and to study its asymptotic behavior with respect to the partial-transport penalty $\lambda > 0$ and the entropy parameter $\eps > 0$. We first formulate a finite EPOT problem by combining the dummy-point representation of partial transport with entropic regularization. For fixed $\eps > 0$, we prove that this model recovers balanced entropic OT as $\lambda \to \infty$, with quantitative estimates for the transported mass and the value, together with convergence of the optimizer.

We then specialize EPOT to Gaussian mixtures by using squared Gaussian Wasserstein distances as component costs. The resulting partial component coupling induces a continuous Gaussian-mixture transport plan through the optimal Gaussian maps between components. We prove narrow convergence of these induced plans in the large-penalty limit and analyze the subsequent zero-entropy limit, where the entropic optimizer selects the maximum-entropy optimizer among the unregularized mixture OT minimizers.

Finally, we extend the same partial-matching idea to Gromov--Wasserstein (GW) geometry by applying partial GW to the finite metric measure spaces of Gaussian components. This yields a partial mixture GW distance, whose basic properties follow from the partial GW theory. We also derive barycentric projection maps that convert component couplings into pointwise correspondences and illustrate the proposed methods through Gaussian-mixture and point-cloud matching experiments with outliers.

The paper is organized as follows: Section~\ref{sec:preliminaries} recalls partial OT and partial GW, together with the notation used throughout the paper. Section~\ref{sec:epot} develops the finite EPOT theory and proves convergence of values and optimizers in the large-penalty limit. Section~\ref{sec:gmm-epot} specializes EPOT to Gaussian mixture models, proves narrow convergence of the induced continuous transport plans, and analyzes the zero-entropy limit. Section~\ref{sec:pmgw} defines partial mixture GW distances and establishes their basic properties. Section~\ref{sec:barycentric} introduces barycentric projection maps for pointwise matching. Section~\ref{sec:numerical} presents numerical experiments. Section~\ref{sec:conclusions} concludes the paper.

\section{Preliminaries}\label{sec:preliminaries}

We recall the partial optimal transport and partial Gromov--Wasserstein formulations that underlie the constructions in this paper. For a nonnegative finite Borel measure $\alpha$, we write $|\alpha|$ for its total mass. If $\gamma$ is a measure on a product space, $(\pi_1)_\#\gamma$ and $(\pi_2)_\#\gamma$ denote its first and second marginals. We write $\MM_+(\Z)$ for the set of nonnegative finite Borel measures on a measurable space $\Z$.

\subsection{Partial optimal transport}\label{subsec:pot}

Let $(\X,\mu)$ and $(\Y,\nu)$ be Polish spaces equipped with Borel probability measures, and let $c:\X\times\Y\to\RR_{\geq 0}$ be a measurable cost function. Recall that $\Pi(\mu,\nu)$ denotes the set of probability measures on $\X\times\Y$ with marginals $\mu$ and $\nu$, as in the balanced Kantorovich problem \eqref{eq:intro-kantorovich}. The partial transport feasible set is
\begin{equation*}
  \Pi_{ \leq }(\mu,\nu) = \{\gamma \in \MM_+(\X\times\Y): (\pi_1)_\#\gamma \leq \mu,\; (\pi_2)_\#\gamma \leq \nu\}.
\end{equation*}
For a penalty parameter $\lambda > 0$, the partial optimal transport problem is
\begin{equation}\label{eq:pot-direct}
  \POT^{\lambda}(\mu,\nu) = \inf_{\gamma \in \Pi_{ \leq }(\mu,\nu)} \left\{\int_{\X\times\Y} c(x,y)\dd\gamma(x,y) + \lambda\left(|\mu-(\pi_1)_\#\gamma| + |\nu-(\pi_2)_\#\gamma|\right)\right\}.
\end{equation}
Since $\gamma \in \Pi_{\leq}(\mu,\nu)$, the measures
$\mu-(\pi_1)_\#\gamma$ and $\nu-(\pi_2)_\#\gamma$ are nonnegative.
As pushforwards preserve total mass and $\mu,\nu$ are probability measures, we have
\begin{equation*}
  |\mu-(\pi_1)_\#\gamma| = |\nu-(\pi_2)_\#\gamma| = 1-|\gamma|.
\end{equation*}
Thus the penalty term in \eqref{eq:pot-direct} is equal to $2\lambda(1-|\gamma|)$.

We use the following equivalent dummy-point formulation. Set
\begin{equation*}
  \widetilde\X = \X\cup\{\infty_{\X}\},\qquad \widetilde\Y = \Y\cup\{\infty_{\Y}\},\qquad \widetilde\mu = \mu + \delta_{\infty_{\X}},\qquad \widetilde\nu = \nu + \delta_{\infty_{\Y}}.
\end{equation*}
We define the extended cost by
\begin{equation*}
  \widetilde c(x,y) = \begin{cases} c(x,y) - 2\lambda, & (x,y)\in\X\times\Y,\\ 0, & \text{otherwise}. \end{cases}
\end{equation*}
Then the extended problem is
\begin{equation}\label{eq:pot-extended}
  \inf_{\widetilde\gamma \in \Pi(\widetilde\mu,\widetilde\nu)} \left\{\int_{\widetilde\X\times\widetilde\Y} \widetilde c(x,y)\dd\widetilde\gamma(x,y) + 2\lambda\right\}.
\end{equation}
The dummy-point formulation rewrites a partial coupling as a balanced coupling on the enlarged spaces. If $\gamma \in \Pi_{\leq}(\mu,\nu)$, its unmatched marginals are $\mu - (\pi_1)_\#\gamma$ and $\nu - (\pi_2)_\#\gamma$. Assigning these unmatched marginals to the dummy points defines the completion map $T:\Pi_{\leq}(\mu,\nu)\to\Pi(\widetilde\mu,\widetilde\nu)$ by
\begin{equation*}
  T(\gamma) = \gamma + \bigl(\mu - (\pi_1)_\#\gamma\bigr)\otimes\delta_{\infty_{\Y}} + \delta_{\infty_{\X}}\otimes\bigl(\nu - (\pi_2)_\#\gamma\bigr) + |\gamma|\delta_{(\infty_{\X},\infty_{\Y})}.
\end{equation*}
Then $T(\gamma)\in\Pi(\widetilde\mu,\widetilde\nu)$. Conversely, the restriction of any $\widetilde\gamma \in \Pi(\widetilde\mu,\widetilde\nu)$ to $\X\times\Y$ belongs to $\Pi_{\leq}(\mu,\nu)$.

\begin{lem}\label{lem:pot-to-ot}
For every $\gamma \in \Pi_{\leq}(\mu,\nu)$, one has
\begin{equation*}
  \int_{\widetilde\X\times\widetilde\Y} \widetilde c(x,y)\dd T(\gamma)(x,y) + 2\lambda = \int_{\X\times\Y} c(x,y)\dd\gamma(x,y) + \lambda\left(|\mu-(\pi_1)_\#\gamma| + |\nu-(\pi_2)_\#\gamma|\right).
\end{equation*}
Consequently, minimizers of \eqref{eq:pot-direct} and \eqref{eq:pot-extended} correspond through $T$ and restriction to $\X\times\Y$.
\end{lem}

\begin{proof}
By the definition of $\widetilde c$, only the real--real part contributes to the integral. Hence
\begin{equation}\label{eq:partial-to-balanced}
  \int_{\widetilde\X\times\widetilde\Y} \widetilde c\dd T(\gamma) + 2\lambda = \int_{\X\times\Y} (c - 2\lambda)\dd\gamma + 2\lambda = \int_{\X\times\Y} c\dd\gamma + 2\lambda(1-|\gamma|).
\end{equation}
Since $\gamma \in \Pi_{\leq}(\mu,\nu)$ and $\mu,\nu$ are probability measures,
\begin{equation*}
  |\mu-(\pi_1)_\#\gamma| = |\nu-(\pi_2)_\#\gamma| = 1-|\gamma|.
\end{equation*}
Combining this with \eqref{eq:partial-to-balanced} gives the stated identity. The minimizer correspondence follows by applying this identity to $T(\gamma)$ and by restricting extended couplings to $\X\times\Y$.
\end{proof}

\begin{rem}\label{rem:lambda-interpretation}
The shift by $-2\lambda$ gives a reward of $2\lambda$ per unit of mass transported between real points. Thus $\lambda$ controls the selectivity of the partial matching. Small values of $\lambda$ allow costly components to remain unmatched, whereas large values of $\lambda$ encourage more real mass to be transported.
\end{rem}

The following elementary completion lemma will be used below. It says that any partial real--real coupling can be completed to a balanced coupling by adding a product coupling of the residual marginals.

\begin{lem}\label{lem:dominant-coupling}
Let $\widetilde\gamma \in \Pi(\widetilde\mu,\widetilde\nu)$ and put $\gamma^0 = \widetilde\gamma|_{\X\times\Y}$. If $|\gamma^0| < 1$, we define
\begin{equation*}
  \mu^r = \mu - (\pi_1)_\#\gamma^0, \qquad \nu^r = \nu - (\pi_2)_\#\gamma^0.
\end{equation*}
Then
\begin{equation}\label{eq:completion-coupling}
  \overline\gamma = \gamma^0 + \frac{1}{1 - |\gamma^0|}\mu^r\otimes\nu^r
\end{equation}
belongs to $\Pi(\mu,\nu)$. In particular, $\overline\gamma$ dominates $\gamma^0$ as a measure on $\X\times\Y$.
\end{lem}

\begin{proof}
Since $\gamma^0 \in \Pi_{\leq}(\mu,\nu)$, the residual measures $\mu^r$ and $\nu^r$ are nonnegative and satisfy
\begin{equation*}
  |\mu^r| = |\nu^r| = 1 - |\gamma^0|.
\end{equation*}
The first marginal of $\mu^r\otimes\nu^r$ is $|\nu^r|\mu^r$, and its second marginal is $|\mu^r|\nu^r$. Hence, the product term in \eqref{eq:completion-coupling} has first marginal $\mu^r$ and second marginal $\nu^r$. Therefore $\overline\gamma$ has marginals $\mu$ and $\nu$, which proves $\overline\gamma \in \Pi(\mu,\nu)$.
\end{proof}

\begin{thm}\label{thm:lambda-to-infty}
Assume that $c$ is bounded and that
$\lambda > \frac12\|c\|_{L^\infty(\X\times\Y)}$. Then
\begin{equation*}
  \POT^{\lambda}(\mu,\nu) = \OT(\mu,\nu).
\end{equation*}
Moreover, any minimizer $\gamma$ of \eqref{eq:pot-direct}, if it exists, satisfies $|\gamma|=1$ and hence belongs to $\Pi(\mu,\nu)$.
\end{thm}

\begin{proof}
Since $\Pi(\mu,\nu)\subset \Pi_{\leq}(\mu,\nu)$ and the penalty term in \eqref{eq:pot-direct} vanishes on $\Pi(\mu,\nu)$, we have
\begin{equation*}
  \POT^{\lambda}(\mu,\nu) \leq \OT(\mu,\nu).
\end{equation*}

We prove the reverse inequality. Let $\gamma \in \Pi_{\leq}(\mu,\nu)$. If $|\gamma|=1$, then $\gamma \in \Pi(\mu,\nu)$ and
\begin{equation*}
  \int_{\X\times\Y} c\dd\gamma + 2\lambda(1-|\gamma|) = \int_{\X\times\Y} c\dd\gamma \geq \OT(\mu,\nu).
\end{equation*}
Suppose that $|\gamma|<1$. Applying Lemma~\ref{lem:dominant-coupling} to $T(\gamma)$, we obtain $\overline\gamma \in \Pi(\mu,\nu)$ such that $\overline\gamma \geq \gamma$ and $|\overline\gamma-\gamma|=1-|\gamma|$. 
Hence,
\begin{equation*}
  \int_{\X\times\Y} c\dd\overline\gamma = \int_{\X\times\Y} c\dd\gamma + \int_{\X\times\Y} c\dd(\overline\gamma-\gamma) \leq \int_{\X\times\Y} c\dd\gamma + \|c\|_{L^\infty(\X\times\Y)}(1-|\gamma|) < \int_{\X\times\Y} c\dd\gamma + 2\lambda(1-|\gamma|).
\end{equation*}
Since $\overline\gamma \in \Pi(\mu,\nu)$, the left-hand side is at least $\OT(\mu,\nu)$. Therefore
\begin{equation*}
  \int_{\X\times\Y} c\dd\gamma + 2\lambda(1-|\gamma|) > \OT(\mu,\nu)
\end{equation*}
whenever $|\gamma|<1$. Together with the case $|\gamma|=1$, this gives
\begin{equation*}
  \POT^{\lambda}(\mu,\nu) \geq \OT(\mu,\nu).
\end{equation*}
Thus $\POT^{\lambda}(\mu,\nu)=\OT(\mu,\nu)$. The strict inequality also shows that no minimizer can have $|\gamma|<1$.
\end{proof}

\begin{rem}\label{rem:bounded-cost}
The boundedness assumption is used only to obtain the uniform threshold
$\lambda > \frac12\|c\|_{L^\infty(\X\times\Y)}$. In the finite component problems considered below, the cost is a finite matrix, so this assumption is automatic.
\end{rem}

\subsection{Gromov--Wasserstein and partial Gromov--Wasserstein distances}\label{subsec:gw}

Let $(\X,d_{\X},\mu)$ and $(\Y,d_{\Y},\nu)$ be metric measure spaces. For $p,q \geq 1$, the balanced Gromov--Wasserstein distance is
\begin{align}\label{eq:gw-renamed}
  &\left(\GW_{p,q}\bigl((\X,d_{\X},\mu),(\Y,d_{\Y},\nu)\bigr)\right)^p \notag \\
  &= \inf_{\gamma \in \Pi(\mu,\nu)} \left\{\int_{(\X\times\Y)^2}|d_{\X}(x,x')^q-d_{\Y}(y,y')^q|^p\dd\gamma(x,y)\dd\gamma(x',y')\right\}.
\end{align}
This distance compares intrinsic distance structures rather than point locations, and hence can be applied even when the two spaces have different ambient dimensions or are observed up to isometry. We use the notation $\GW^{\lambda}_{p,q}$ for the partial Gromov--Wasserstein distance, where the superscript $\lambda$ denotes the partial matching penalty. Following Bai et al.~\cite{Bai2025}, we define
\begin{align}\label{eq:pgw-renamed}
  &\left(\GW^{\lambda}_{p,q}\bigl((\X,d_{\X},\mu),(\Y,d_{\Y},\nu)\bigr)\right)^p \notag \\
  &= \inf_{\gamma \in \Pi_{\leq}(\mu,\nu)} \left\{\int_{(\X\times\Y)^2}\left(|d_{\X}(x,x')^q-d_{\Y}(y,y')^q|^p-2\lambda\right)\dd\gamma(x,y)\dd\gamma(x',y') + \lambda\bigl(|\mu|^2+|\nu|^2\bigr) \right\}.
\end{align}
The following result collects the basic properties of the partial Gromov--Wasserstein distance proved by Bai et al.~\cite{Bai2025}, rewritten in the notation of \eqref{eq:pgw-renamed}.

\begin{thm}[Bai et al.~\cite{Bai2025}]\label{thm:gw-lambda}
The partial Gromov--Wasserstein distance defined by \eqref{eq:pgw-renamed} has the following properties: 
\begin{enumerate}[label=\textup{(\roman*)}]
  \item $\GW^{\lambda}_{p,q}$ admits a minimizer.
  \item $\GW^{\lambda}_{p,q}$ defines a metric on metric measure spaces modulo strong isomorphism.
  \item If the supports are compact and
  \begin{equation*}
  \lambda \geq \max_{x,x' \in \supp(\mu),\;y,y' \in \supp(\nu)} |d_{\X}(x,x')^q-d_{\Y}(y,y')^q|^p,
\end{equation*}
then the partial distance \eqref{eq:pgw-renamed} coincides with the balanced distance \eqref{eq:gw-renamed}. In particular, for finite component spaces, $\GW^{\lambda}_{p,q}$ converges to $\GW_{p,q}$ as $\lambda \to \infty$.
\end{enumerate}
\end{thm}

\section{Entropic Partial Optimal Transport}\label{sec:epot}

This section develops the finite entropic partial optimal transport problem and analyzes its large-penalty limit. This is the setting needed for Gaussian mixture components. Throughout this section, we use the convention $0\log 0 = 0$.

\subsection{Discrete formulation}\label{subsec:epot-discrete}

Let $\X = \{1,\ldots,n\}$ and $\Y = \{1,\ldots,m\}$. Let
\begin{equation*}
  a = (a_1,\ldots,a_n),\qquad b = (b_1,\ldots,b_m)
\end{equation*}
be probability vectors with positive entries. We identify them with the measures $\mu = \sum_{i=1}^n a_i\delta_i$ and $\nu = \sum_{j=1}^m b_j\delta_j$. For nonnegative vectors $r$ and $s$ with equal total mass, we denote by $\Pi(r,s)$ the set of nonnegative matrices with row sums $r$ and column sums $s$. Let $C = (c_{ij})\in \RR^{n\times m}_{\geq 0}$ be the cost matrix. The extended marginal vectors are
\begin{equation*}
  \widetilde a = (a_1,\ldots,a_n,1),\qquad \widetilde b = (b_1,\ldots,b_m,1).
\end{equation*}
For $\lambda>0$, we define the extended cost matrix $\widetilde C^{\lambda}\in\RR^{(n+1)\times(m+1)}$ by
\begin{equation*}
  \widetilde c^{\lambda}_{ij} = \begin{cases} c_{ij} - 2\lambda, & 1 \leq i \leq n \text{ and } 1 \leq j \leq m,\\ 0, & \text{otherwise.} \end{cases}
\end{equation*}
For a nonnegative matrix $A=(A_{ij})$, we define
\begin{equation*}
  E(A) = \sum_{i,j}A_{ij}(\log A_{ij} - 1).
\end{equation*}

\begin{dfn}\label{dfn:epot}
For $\eps > 0$ and $\lambda > 0$, the entropic partial optimal transport problem is defined by
\begin{equation}\label{eq:epot}
  \OT_{\eps,\lambda}(a,b;C) = \min_{\Gamma \in \Pi(\widetilde a,\widetilde b)} \left\{\sum_{i = 1}^{n + 1}\sum_{j = 1}^{m + 1}\widetilde c^{\lambda}_{ij}\Gamma_{ij} + \eps E(\Gamma) + 2\lambda\right\}.
\end{equation}
The balanced entropic optimal transport problem is defined by
\begin{equation}\label{eq:eot}
  \OT_{\eps}(a,b;C) = \min_{\omega \in \Pi(a,b)} \left\{\sum_{i = 1}^{n}\sum_{j = 1}^{m}c_{ij}\omega_{ij} + \eps E(\omega)\right\}.
\end{equation}
\end{dfn}
Throughout this section, the cost matrix $C$ is fixed, and we write $\OT_{\eps,\lambda}(a,b)$ and $\OT_{\eps}(a,b)$ for the values in \eqref{eq:epot} and \eqref{eq:eot}, respectively.

\begin{prop}\label{prop:exist-unique}
For every $\eps > 0$ and $\lambda > 0$, the problem \eqref{eq:epot} admits a unique minimizer, denoted by $\Gamma^{\eps,\lambda}$.
\end{prop}

\begin{proof}
The feasible set $\Pi(\widetilde a,\widetilde b)$ is a nonempty compact convex polytope in $\RR^{(n+1)\times(m+1)}$. The function $x\mapsto x(\log x-1)$ is continuous and strictly convex on $[0,\infty)$. Hence, $E$ is continuous and strictly convex on $\Pi(\widetilde a,\widetilde b)$. Since the cost term in \eqref{eq:epot} is linear, the objective is continuous and strictly convex. Thus, existence follows from compactness, and uniqueness follows from strict convexity.
\end{proof}

We first show that the mass transported between the original points converges to one as $\lambda\to\infty$.

\begin{prop}\label{prop:mass-convergence}
Fix $\eps > 0$. Let $\Gamma^{\eps,\lambda}$ be the unique minimizer of \eqref{eq:epot}, and define
\begin{equation*}
  Z_{\lambda} = \sum_{i=1}^{n}\sum_{j=1}^{m}\Gamma^{\eps,\lambda}_{ij}.
\end{equation*}
Then
\begin{equation}\label{eq:zlambda-to-one}
  Z_{\lambda}\to 1\qquad\text{as }\lambda \to \infty.
\end{equation}
More precisely, if $M = \displaystyle \max_{1\leq i\leq n,\;1\leq j\leq m}c_{ij}$ and $2\lambda > M$, then
\begin{equation}\label{eq:mass-rate}
  0 \leq 1 - Z_{\lambda} \leq \frac{\eps(n + 1)(m + 1)}{2\lambda - M}.
\end{equation}
\end{prop}

\begin{proof}
Let $\Gamma = \Gamma^{\eps,\lambda}$ and let $\Gamma^0$ be its $n\times m$ real--real block. Since $\Gamma\in\Pi(\widetilde a,\widetilde b)$,
\begin{equation*}
  0 \leq Z_{\lambda}=\sum_{i=1}^{n}\sum_{j=1}^{m}\Gamma^0_{ij} \leq \sum_{i=1}^{n}a_i=1.
\end{equation*}
If $Z_{\lambda}=1$, there is nothing to prove. Suppose that $Z_{\lambda}<1$, and set
\begin{equation*}
  r_i = a_i-\sum_{j=1}^{m}\Gamma^0_{ij}=\Gamma_{i,m+1},\qquad s_j = b_j-\sum_{i=1}^{n}\Gamma^0_{ij}=\Gamma_{n+1,j}.
\end{equation*}
Then $r_i\geq0$, $s_j\geq0$, and
\begin{equation*}
  \sum_{i=1}^{n}r_i=\sum_{j=1}^{m}s_j=1-Z_{\lambda}.
\end{equation*}
Define
\begin{equation*}
  \Delta_{ij}=\frac{r_i s_j}{1-Z_{\lambda}},\qquad \overline\Gamma^0_{ij}=\Gamma^0_{ij}+\Delta_{ij}.
\end{equation*}
Since
\begin{equation*}
  \sum_{j=1}^{m}\Delta_{ij}=r_i,\qquad \sum_{i=1}^{n}\Delta_{ij}=s_j,
\end{equation*}
the matrix $\overline\Gamma^0$ has row sums $a$ and column sums $b$. Therefore,
\begin{equation*}
  \overline\Gamma =
  \begin{pmatrix}
    \overline\Gamma^0 & 0\\
    0^{\top} & 1
  \end{pmatrix}
  \in\Pi(\widetilde a,\widetilde b).
\end{equation*}
Moreover,
\begin{equation*}
  \sum_{i=1}^{n}\sum_{j=1}^{m}\Delta_{ij}=1-Z_{\lambda}.
\end{equation*}
Since $\widetilde c^\lambda_{ij}=0$ whenever $i=n+1$ or $j=m+1$, only the real--real block contributes to the change in the linear cost. Hence,
\begin{equation}\label{eq:mass-linear-estimate}
  \sum_{i = 1}^{n + 1}\sum_{j = 1}^{m + 1}\widetilde c^\lambda_{ij}(\overline\Gamma_{ij}-\Gamma_{ij}) = \sum_{i = 1}^{n}\sum_{j = 1}^{m}(c_{ij}-2\lambda)\Delta_{ij} \leq (M-2\lambda)(1-Z_{\lambda}).
\end{equation}

For any $\Theta\in\Pi(\widetilde a,\widetilde b)$, the marginal constraints and nonnegativity give
\begin{equation*}
  0 \leq \Theta_{ij} \leq \min\{\widetilde a_i,\widetilde b_j\} \leq 1.
\end{equation*}
Since $x(\log x-1) \in[-1,0]$ for $x\in[0,1]$, we have
\begin{equation*}
  -(n+1)(m+1) \leq E(\Theta) \leq 0.
\end{equation*}
In particular,
\begin{equation}\label{eq:entropy-estimate}
  E(\overline\Gamma)-E(\Gamma) \leq (n + 1)(m + 1).
\end{equation}
Since $\Gamma$ is a minimizer of \eqref{eq:epot} and $\overline\Gamma\in\Pi(\widetilde a,\widetilde b)$, we obtain
\begin{equation}\label{eq:mass-optimality}
  0 \leq \sum_{i = 1}^{n + 1}\sum_{j = 1}^{m + 1}\widetilde c^\lambda_{ij}(\overline\Gamma_{ij}-\Gamma_{ij}) + \eps\bigl(E(\overline\Gamma)-E(\Gamma)\bigr).
\end{equation}
Combining \eqref{eq:mass-linear-estimate}, \eqref{eq:entropy-estimate}, and \eqref{eq:mass-optimality} gives
\begin{equation*}
  0 \leq (M-2\lambda)(1-Z_{\lambda})+\eps(n+1)(m+1).
\end{equation*}
Therefore,
\begin{equation*}
  (2\lambda-M)(1-Z_{\lambda}) \leq \eps(n+1)(m+1),
\end{equation*}
which gives \eqref{eq:mass-rate}. Thus, the convergence \eqref{eq:zlambda-to-one} follows.
\end{proof}

We next compare the entropic partial value with the balanced entropic OT value. We continue with the notation of Proposition~\ref{prop:mass-convergence} and set
\begin{equation*}
  \rho_{\lambda}=1-Z_{\lambda}.
\end{equation*}

\begin{thm}\label{thm:value-error-estimate}
The following estimate holds: 
\begin{equation}\label{eq:value-error-estimate}
  0 \leq \OT_{\eps}(a,b)-\eps-\OT_{\eps,\lambda}(a,b) \leq M\rho_{\lambda} + \eps\rho_{\lambda}\left(2|\log\rho_{\lambda}| + \log(nm) + 2\right),
\end{equation}
where $\rho_{\lambda}|\log\rho_{\lambda}|$ is understood to be zero when $\rho_{\lambda}=0$. Consequently,
\begin{equation*}
  \OT_{\eps}(a,b)-\eps-\OT_{\eps,\lambda}(a,b) = O\left(\frac{\log\lambda}{\lambda}\right) \qquad\text{as }\lambda\to\infty.
\end{equation*}
\end{thm}

\begin{proof}
Let $\omega^\eps$ be the minimizer of \eqref{eq:eot}, and define
\begin{equation*}
  \widetilde\omega^\eps =
  \begin{pmatrix}
    \omega^\eps & 0\\
    0^\top & 1
  \end{pmatrix}.
\end{equation*}
Then $\widetilde\omega^\eps\in\Pi(\widetilde a,\widetilde b)$. Since $\omega^\eps$ has total mass one and $E(\widetilde\omega^\eps)=E(\omega^\eps)-1$, substituting $\widetilde\omega^\eps$ into \eqref{eq:epot} gives
\begin{equation*}
  \sum_{i = 1}^{n + 1}\sum_{j = 1}^{m + 1}\widetilde c^\lambda_{ij}\widetilde\omega^\eps_{ij} + \eps E(\widetilde\omega^\eps) + 2\lambda = \sum_{i = 1}^{n}\sum_{j = 1}^{m}c_{ij}\omega^\eps_{ij} + \eps E(\omega^\eps)-\eps.
\end{equation*}
Therefore,
\begin{equation}\label{eq:epot-upper-balanced}
  \OT_{\eps,\lambda}(a,b) \leq \OT_{\eps}(a,b)-\eps.
\end{equation}
This proves the first inequality in \eqref{eq:value-error-estimate}.

Let $\Gamma=\Gamma^{\eps,\lambda}$ and let $\Gamma^0$ be its $n\times m$ real--real block. If $\rho_{\lambda}=0$, then $\Gamma^0\in\Pi(a,b)$, all real--dummy and dummy--real entries vanish, and $\Gamma_{n+1,m+1}=1$. Hence,
\begin{equation*}
  \OT_{\eps,\lambda}(a,b) = \sum_{i = 1}^{n}\sum_{j = 1}^{m}c_{ij}\Gamma^0_{ij} + \eps E(\Gamma^0)-\eps \geq \OT_{\eps}(a,b)-\eps.
\end{equation*}
Together with \eqref{eq:epot-upper-balanced}, this gives equality. We may therefore assume that $\rho_{\lambda}>0$.
Set
\begin{equation*}
  r_i=a_i-\sum_{j=1}^{m}\Gamma^0_{ij},\qquad s_j=b_j-\sum_{i=1}^{n}\Gamma^0_{ij}.
\end{equation*}
The marginal constraints on $\Gamma$ imply
\begin{equation*}
  r_i=\Gamma_{i,m+1}\geq0,\qquad s_j=\Gamma_{n+1,j}\geq0,
\end{equation*}
and
\begin{equation*}
  \sum_{i=1}^{n}r_i=\sum_{j=1}^{m}s_j=\rho_{\lambda}.
\end{equation*}
Define
\begin{equation*}
  \Delta_{ij}=\frac{r_i s_j}{\rho_{\lambda}},\qquad \overline\omega_{ij}=\Gamma^0_{ij}+\Delta_{ij}.
\end{equation*}
Since
\begin{equation*}
  \sum_{j=1}^{m}\Delta_{ij}=r_i,\qquad \sum_{i=1}^{n}\Delta_{ij}=s_j,
\end{equation*}
the matrix $\overline\omega$ has row sums $a$ and column sums $b$. Thus,
\begin{equation*}
  \overline\omega\in\Pi(a,b).
\end{equation*}
Moreover,
\begin{equation*}
  \sum_{i = 1}^{n}\sum_{j = 1}^{m}\Delta_{ij} = \frac{\left(\sum_{i = 1}^{n}r_i\right)\left(\sum_{j = 1}^{m}s_j\right)}{\rho_{\lambda}} = \rho_{\lambda}.
\end{equation*}
It follows that
\begin{equation}\label{eq:value-cost-estimate}
  \sum_{i = 1}^{n}\sum_{j = 1}^{m}c_{ij}\overline\omega_{ij} = \sum_{i = 1}^{n}\sum_{j = 1}^{m}c_{ij}\Gamma^0_{ij} + \sum_{i = 1}^{n}\sum_{j = 1}^{m}c_{ij}\Delta_{ij} \leq \sum_{i = 1}^{n}\sum_{j = 1}^{m}c_{ij}\Gamma^0_{ij} + M\rho_{\lambda}.
\end{equation}

Let $f(t)=t(\log t-1)$. Since $f'(t)=\log t\leq0$ for $t\in(0,1]$, the function $f$ is decreasing on $[0,1]$. Furthermore,
\begin{equation*}
  0\leq\Gamma^0_{ij}\leq\overline\omega_{ij}\leq\min\{a_i,b_j\}\leq1.
\end{equation*}
Therefore,
\begin{equation}\label{eq:value-entropy-estimate}
  E(\overline\omega) \leq E(\Gamma^0).
\end{equation}
Since $\overline\omega\in\Pi(a,b)$, the definition of $\OT_{\eps}(a,b)$ gives
\begin{equation*}
  \OT_{\eps}(a,b) \leq \sum_{i = 1}^{n}\sum_{j = 1}^{m}c_{ij}\overline\omega_{ij} + \eps E(\overline\omega).
\end{equation*}
Combining this inequality with \eqref{eq:value-cost-estimate} and \eqref{eq:value-entropy-estimate}, we obtain
\begin{equation}\label{eq:balanced-value-estimate}
  \OT_{\eps}(a,b) \leq \sum_{i = 1}^{n}\sum_{j = 1}^{m}c_{ij}\Gamma^0_{ij} + M\rho_{\lambda} + \eps E(\Gamma^0).
\end{equation}
The marginal constraints on $\Gamma$ also give
\begin{equation*}
  \Gamma_{i,m+1}=r_i,\qquad \Gamma_{n+1,j}=s_j,\qquad \Gamma_{n+1,m+1}=Z_{\lambda}.
\end{equation*}
Define
\begin{equation*}
  E_{\mathrm{dum}}(\Gamma) = \sum_{i = 1}^{n}r_i(\log r_i-1) + \sum_{j = 1}^{m}s_j(\log s_j-1) + Z_{\lambda}(\log Z_{\lambda}-1).
\end{equation*}
Since the dummy costs vanish and the real--real mass of $\Gamma$ is $Z_{\lambda}=1-\rho_{\lambda}$, we have
\begin{equation}\label{eq:epot-block-decomposition}
  \OT_{\eps,\lambda}(a,b) = \sum_{i = 1}^{n}\sum_{j = 1}^{m}c_{ij}\Gamma^0_{ij} + \eps E(\Gamma^0) + \eps E_{\mathrm{dum}}(\Gamma) + 2\lambda\rho_{\lambda}.
\end{equation}
Combining \eqref{eq:balanced-value-estimate} and \eqref{eq:epot-block-decomposition} gives
\begin{equation*}
  \OT_{\eps}(a,b)-\eps-\OT_{\eps,\lambda}(a,b) \leq M\rho_{\lambda} + \eps\bigl(-E_{\mathrm{dum}}(\Gamma)-1\bigr) - 2\lambda\rho_{\lambda}.
\end{equation*}
Since $\lambda\rho_{\lambda}\geq0$, it follows that
\begin{equation}\label{eq:value-gap-entropy}
  \OT_{\eps}(a,b)-\eps-\OT_{\eps,\lambda}(a,b) \leq M\rho_{\lambda} + \eps\bigl(-E_{\mathrm{dum}}(\Gamma)-1\bigr).
\end{equation}
Using
\begin{equation*}
  \sum_{i=1}^{n}r_i=\sum_{j=1}^{m}s_j=\rho_{\lambda},\qquad Z_{\lambda}=1-\rho_{\lambda},
\end{equation*}
we obtain
\begin{equation}\label{eq:estimate-Edum1}
  -E_{\mathrm{dum}}(\Gamma)-1 = -\sum_{i = 1}^{n}r_i\log r_i - \sum_{j = 1}^{m}s_j\log s_j - Z_{\lambda}\log Z_{\lambda} + \rho_{\lambda} \geq 0.
\end{equation}
Since $(r_i/\rho_{\lambda})_{i=1}^{n}$ is a probability vector and its entropy is at most $\log n$, we obtain 
\begin{equation}\label{eq:estimate-entropy-r}
  -\sum_{i = 1}^{n}r_i\log r_i = -\rho_{\lambda}\log\rho_{\lambda} - \rho_{\lambda}\sum_{i = 1}^{n}\frac{r_i}{\rho_{\lambda}}\log\frac{r_i}{\rho_{\lambda}} \leq \rho_{\lambda}\log\frac{n}{\rho_{\lambda}}.
\end{equation}
Similarly,
\begin{equation}\label{eq:estimate-entropy-s}
  -\sum_{j = 1}^{m}s_j\log s_j \leq \rho_{\lambda}\log\frac{m}{\rho_{\lambda}}.
\end{equation}
Moreover, the inequality $-t\log t\leq1-t$ on $[0,1]$ gives
\begin{equation}\label{eq:estimate-entropy-Z}
  -Z_{\lambda}\log Z_{\lambda}\leq1-Z_{\lambda}=\rho_{\lambda}.
\end{equation}
Since $0<\rho_{\lambda}\leq1$, we have $-\log\rho_{\lambda}=|\log\rho_{\lambda}|$. Therefore, applying estimates \eqref{eq:estimate-entropy-r}, \eqref{eq:estimate-entropy-s}, and \eqref{eq:estimate-entropy-Z} to \eqref{eq:estimate-Edum1}, we obtain 
\begin{equation}\label{eq:dummy-entropy-estimate}
  0 \leq -E_{\mathrm{dum}}(\Gamma)-1 \leq \rho_{\lambda}\left(2|\log\rho_{\lambda}| + \log(nm) + 2\right).
\end{equation}
Combining \eqref{eq:value-gap-entropy} and \eqref{eq:dummy-entropy-estimate}, we prove \eqref{eq:value-error-estimate}.

Finally, Proposition~\ref{prop:mass-convergence} gives
\begin{equation}\label{eq:mass-estimate}
  0\leq\rho_{\lambda}\leq\frac{\eps(n+1)(m+1)}{2\lambda-M}
\end{equation}
for all sufficiently large $\lambda$. Hence,
\begin{equation*}
  \rho_{\lambda}=O\left(\frac{1}{\lambda}\right) \qquad\text{as }\lambda\to\infty.
\end{equation*}
For sufficiently large $\lambda$, the right-hand side of \eqref{eq:mass-estimate} is less than $e^{-1}$. Since $t|\log t|$ is increasing on $[0,e^{-1}]$, we also have
\begin{equation*}
  \rho_{\lambda}|\log\rho_{\lambda}|=O\left(\frac{\log\lambda}{\lambda}\right) \qquad\text{as }\lambda\to\infty.
\end{equation*}
The asymptotic estimate follows from \eqref{eq:value-error-estimate}.
\end{proof}

\begin{cor}\label{cor:value-convergence}
For every fixed $\eps>0$,
\begin{equation}\label{eq:value-convergence}
  \lim_{\lambda\to\infty}\OT_{\eps,\lambda}(a,b) = \OT_{\eps}(a,b)-\eps.
\end{equation}
\end{cor}

\begin{proof}
By Theorem~\ref{thm:value-error-estimate},
\begin{equation*}
  0 \leq \OT_{\eps}(a,b) - \OT_{\eps,\lambda}(a,b) -\eps \leq M\rho_{\lambda}+\eps\rho_{\lambda}\left(2|\log\rho_{\lambda}|+\log(nm)+2\right).
\end{equation*}
Proposition~\ref{prop:mass-convergence} gives $\rho_{\lambda}\to0$, and hence the right-hand side converges to zero.
\end{proof}

\begin{cor}\label{cor:plan-convergence}
Let $\Gamma^{\eps,\lambda}$ be the minimizer of \eqref{eq:epot}, and let $\omega^\eps$ be the unique minimizer of \eqref{eq:eot}. Then
\begin{equation*}
  \left\|\Gamma^{\eps,\lambda}|_{\X\times\Y}-\omega^\eps\right\|_{\ell^1}\to0
  \qquad\text{as }\lambda\to\infty.
\end{equation*}
\end{cor}

\begin{proof}
Let $\lambda_k\to\infty$ be an arbitrary sequence, set $\Gamma_k=\Gamma^{\eps,\lambda_k}$, and let $\Gamma_k^0$ be the $n\times m$ real--real block of $\Gamma_k$. We show that
\begin{equation*}
  \Gamma_k^0\to\omega^\eps
  \qquad\text{in }\ell^1.
\end{equation*}
Since $\Gamma_k\in\Pi(\widetilde a,\widetilde b)$, the marginal constraints and nonnegativity give
\begin{equation*}
  0\leq(\Gamma_k^0)_{ij}\leq\min\{a_i,b_j\}\leq1.
\end{equation*}
Hence, $(\Gamma_k^0)_k$ is contained in the compact set $[0,1]^{n\times m}$. Let $\overline\omega$ be an arbitrary cluster point of $(\Gamma_k^0)_k$. After passing to a subsequence, which we do not relabel, we have
\begin{equation*}
  \Gamma_k^0\to\overline\omega
  \qquad\text{in }\ell^1.
\end{equation*}

We first show that $\overline\omega$ has marginals $a$ and $b$. For every $i$ and $j$, the marginal constraints on $\Gamma_k$ give
\begin{equation*}
  a_i-\sum_{\ell=1}^{m}(\Gamma_k^0)_{i\ell}=(\Gamma_k)_{i,m+1}\geq0,\qquad b_j-\sum_{\ell=1}^{n}(\Gamma_k^0)_{\ell j}=(\Gamma_k)_{n+1,j}\geq0.
\end{equation*}
Moreover,
\begin{equation*}
  \sum_{i=1}^{n}(\Gamma_k)_{i,m+1}=\sum_{j=1}^{m}(\Gamma_k)_{n+1,j}=\rho_{\lambda_k}.
\end{equation*}
Since all these entries are nonnegative,
\begin{equation*}
  0\leq(\Gamma_k)_{i,m+1}\leq\rho_{\lambda_k},
  \qquad
  0\leq(\Gamma_k)_{n+1,j}\leq\rho_{\lambda_k}.
\end{equation*}
By Proposition~\ref{prop:mass-convergence}, $\rho_{\lambda_k}\to0$. Passing to the limit in the marginal identities, we obtain
\begin{equation*}
  \sum_{j=1}^{m}\overline\omega_{ij}=a_i,
  \qquad
  \sum_{i=1}^{n}\overline\omega_{ij}=b_j.
\end{equation*}
Therefore,
\begin{equation*}
  \overline\omega\in\Pi(a,b).
\end{equation*}

We next show that $\overline\omega$ minimizes the balanced entropic OT problem. We define
\begin{equation*}
  F(\omega)=\sum_{i=1}^{n}\sum_{j=1}^{m}c_{ij}\omega_{ij}+\eps E(\omega).
\end{equation*}
Since $t\mapsto t(\log t-1)$ is continuous on $[0,1]$, the convergence of $\Gamma_k^0$ gives
\begin{equation*}
  F(\Gamma_k^0)\to F(\overline\omega).
\end{equation*}

By the estimate \eqref{eq:dummy-entropy-estimate}, we have
\begin{equation*}
  0 \leq -E_{\mathrm{dum}}(\Gamma_k)-1 \leq \rho_{\lambda_k}\left(2|\log\rho_{\lambda_k}|+\log(nm)+2\right).
\end{equation*}
Since $\rho_{\lambda_k}\to0$ and $\rho_{\lambda_k}|\log\rho_{\lambda_k}|\to0$, it follows that
\begin{equation}\label{eq:dummy-entropy-convergence}
  E_{\mathrm{dum}}(\Gamma_k)\to-1.
\end{equation}

By the block decomposition \eqref{eq:epot-block-decomposition},
\begin{equation*}
  \OT_{\eps,\lambda_k}(a,b) = F(\Gamma_k^0) + \eps E_{\mathrm{dum}}(\Gamma_k) + 2\lambda_k\rho_{\lambda_k}.
\end{equation*}
Since the last term is nonnegative, we have
\begin{equation*}
  F(\Gamma_k^0)+\eps E_{\mathrm{dum}}(\Gamma_k) \leq \OT_{\eps,\lambda_k}(a,b).
\end{equation*}
Passing to the limit, using the convergence of $F(\Gamma_k^0)$, \eqref{eq:dummy-entropy-convergence}, and Corollary~\ref{cor:value-convergence}, we obtain
\begin{equation*}
  F(\overline\omega)-\eps \leq \OT_{\eps}(a,b)-\eps.
\end{equation*}
Thus, we have 
\begin{equation}\label{eq:F-inequality1}
  F(\overline\omega)\leq\OT_{\eps}(a,b).
\end{equation}
On the other hand, since $\overline\omega\in\Pi(a,b)$, the definition of $\OT_{\eps}(a,b)$ gives
\begin{equation}\label{eq:F-inequality2}
  \OT_{\eps}(a,b)\leq F(\overline\omega).
\end{equation}
Therefore, combining inequalities \eqref{eq:F-inequality1} and \eqref{eq:F-inequality2}, we obtain 
\begin{equation*}
  F(\overline\omega)=\OT_{\eps}(a,b).
\end{equation*}
Hence, $\overline\omega$ is a minimizer of \eqref{eq:eot}. Since $\omega^\eps$ is the unique minimizer of \eqref{eq:eot}, we conclude that
\begin{equation*}
  \overline\omega=\omega^\eps.
\end{equation*}

Since $\overline\omega$ was an arbitrary cluster point, $\omega^\eps$ is the unique cluster point of $(\Gamma_k^0)_k$. Together with the compactness of $[0,1]^{n\times m}$, this implies
\begin{equation*}
  \Gamma_k^0\to\omega^\eps
  \qquad\text{in }\ell^1.
\end{equation*}
Since the sequence $\lambda_k\to\infty$ was arbitrary, the result follows.
\end{proof}

\section{EPOT between Gaussian Mixture Models}\label{sec:gmm-epot}
We now specialize the finite theory to Gaussian mixture models (GMMs). Let
\begin{equation}\label{eq:gmm-pair}
  \mu = \sum_{k=1}^{K}a_k\mu_k,\qquad \nu = \sum_{l=1}^{L}b_l\nu_l
\end{equation}
be Gaussian mixtures on $\RR^d$, where $a_k>0$ and $b_l>0$, and the weights satisfy
\begin{equation*}
  \sum_{k=1}^{K}a_k=\sum_{l=1}^{L}b_l=1.
\end{equation*}
The measures $\mu_k$ and $\nu_l$ are Gaussian probability measures. We write
\begin{equation*}
  a=(a_1,\ldots,a_K),\qquad b=(b_1,\ldots,b_L)
\end{equation*}
for the corresponding weight vectors. If
\begin{equation*}
  \mu_k=\mathcal N(m_{0,k},\Sigma_{0,k}),\qquad
  \nu_l=\mathcal N(m_{1,l},\Sigma_{1,l}),
\end{equation*}
then the squared $2$-Wasserstein distance between $\mu_k$ and $\nu_l$ is given by the classical closed formula \eqref{eq:intro-gaussian-w2}. We use $W_2^2(\mu_k,\nu_l)$ as the cost between the $k$-th source component and the $l$-th target component in the finite EPOT problem \eqref{eq:epot}.

\subsection{Mixture Wasserstein distance and its entropic partial version}\label{subsec:mw}
For the GMMs $\mu$ and $\nu$ in \eqref{eq:gmm-pair}, Delon and Desolneux define a Wasserstein-type distance by restricting continuous transport plans to Gaussian-mixture plans \cite{delon-2020}. In the squared Euclidean case, the continuous formulation \eqref{eq:intro-mw-continuous} is equivalent to the finite-dimensional problem
\begin{equation}\label{eq:mw-discrete}
  \MW_2^2(\mu,\nu)=\min_{\omega\in\Pi(a,b)} \sum_{k=1}^{K}\sum_{l=1}^{L}W_2^2(\mu_k,\nu_l)\omega_{kl}.
\end{equation}
Thus, the continuous mixture transport problem reduces to OT between the mixture weights, with squared Gaussian Wasserstein distances as component costs.

We define the Gaussian component cost matrix by
\begin{equation}\label{eq:gmm-component-cost}
  C_G=(c_{kl})\in\RR_{\geq0}^{K\times L},
  \qquad
  c_{kl}=W_2^2(\mu_k,\nu_l).
\end{equation}
Applying the finite EPOT problem \eqref{eq:epot} to the weight vectors $a$ and $b$ with the cost matrix \eqref{eq:gmm-component-cost}, we define
\begin{equation}\label{eq:gmm-epot-value}
  d_G^{\eps,\lambda}(\mu,\nu)=\OT_{\eps,\lambda}(a,b;C_G).
\end{equation}
We call \eqref{eq:gmm-epot-value} the entropic partial Gaussian mixture OT. The subscript $G$ indicates that the ground cost is the squared Wasserstein distance between Gaussian components.

To define the continuous plans induced by component couplings, assume that all covariance matrices are positive definite. For each $k$, let $p_{\mu_k}$ denote the Gaussian density of $\mu_k$ with respect to the Lebesgue measure. For each pair $(k,l)$, the optimal transport map from $\mu_k$ to $\nu_l$ is the affine map
\begin{equation}\label{eq:gaussian-optimal-map}
  T_{kl}(x)=m_{1,l}+A_{kl}(x-m_{0,k}),\qquad A_{kl}=\Sigma_{0,k}^{-1/2}\left(\Sigma_{0,k}^{1/2}\Sigma_{1,l}\Sigma_{0,k}^{1/2}\right)^{1/2}\Sigma_{0,k}^{-1/2}.
\end{equation}
For any $\omega\in\Pi_{\leq}(a,b)$, we define the induced continuous transport plan by
\begin{equation}\label{eq:gmm-induced-plan}
  \gamma_{\omega}=\sum_{k=1}^{K}\sum_{l=1}^{L}\omega_{kl}(\Id,T_{kl})_\#\mu_k.
\end{equation}
Formally, this plan can be written as
\begin{equation}\label{eq:gmm-induced-plan-formal}
  \gamma_{\omega}(x,y) = \sum_{k=1}^{K}\sum_{l=1}^{L}\omega_{kl}p_{\mu_k}(x)\delta_{y=T_{kl}(x)}.
\end{equation}
Its marginals are
\begin{equation*}
(\pi_1)_\#\gamma_{\omega}=\sum_{k=1}^{K}\left(\sum_{l=1}^{L}\omega_{kl}\right)\mu_k\leq\mu,\qquad (\pi_2)_\#\gamma_{\omega}=\sum_{l=1}^{L}\left(\sum_{k=1}^{K}\omega_{kl}\right)\nu_l\leq\nu.
\end{equation*}
Thus, $\gamma_{\omega}\in\Pi_{\leq}(\mu,\nu)$. In particular, if $\omega\in\Pi(a,b)$, then $\gamma_{\omega}\in\Pi(\mu,\nu)$. For every $\varphi\in C_b(\RR^d\times\RR^d)$, the definition \eqref{eq:gmm-induced-plan} gives
\begin{equation}\label{eq:gmm-induced-plan-test}
  \int_{\RR^d\times\RR^d}\varphi(x,y)\dd\gamma_{\omega}(x,y)=\sum_{k=1}^{K}\sum_{l=1}^{L}\omega_{kl}\int_{\RR^d}\varphi(x,T_{kl}(x))\dd\mu_k(x).
\end{equation}

\begin{thm}\label{thm:gmm-narrow-lambda}
Fix $\eps>0$. Let $\omega^{\eps,\lambda}$ be the real--real block of the EPOT minimizer for the weight vectors $a$ and $b$ with the cost matrix $C_G$, and let $\omega^\eps$ be the unique minimizer of the balanced entropic OT problem with the same cost matrix. We define
\begin{equation}\label{eq:gamma-eps-lambda}
  \gamma^{\eps,\lambda}=\gamma_{\omega^{\eps,\lambda}},\qquad \gamma^\eps=\gamma_{\omega^\eps}.
\end{equation}
Then $\gamma^{\eps,\lambda}$ converges narrowly to $\gamma^\eps$ as $\lambda\to\infty$. More precisely, for every $\varphi\in C_b(\RR^d\times\RR^d)$,
\begin{equation}\label{eq:gmm-narrow-convergence}
  \lim_{\lambda\to\infty}
  \int_{\RR^d\times\RR^d}\varphi(x,y)\dd\gamma^{\eps,\lambda}(x,y) = \int_{\RR^d\times\RR^d}\varphi(x,y)\dd\gamma^\eps(x,y).
\end{equation}
\end{thm}

\begin{proof}
Let $\varphi\in C_b(\RR^d\times\RR^d)$. By \eqref{eq:gmm-induced-plan-test},
\begin{equation*}
  \int_{\RR^d\times\RR^d}\varphi\dd\gamma^{\eps,\lambda}-\int_{\RR^d\times\RR^d}\varphi\dd\gamma^\eps=\sum_{k=1}^{K}\sum_{l=1}^{L}\left(\omega^{\eps,\lambda}_{kl}-\omega^\eps_{kl}\right)\int_{\RR^d}\varphi(x,T_{kl}(x))\dd\mu_k(x).
\end{equation*}
Since $\varphi$ is bounded,
\begin{equation*}
  \left|
  \int_{\RR^d}\varphi(x,T_{kl}(x))\dd\mu_k(x)
  \right|
  \leq
  \|\varphi\|_{\infty}.
\end{equation*}
Therefore,
\begin{equation*}
  \left|\int_{\RR^d\times\RR^d}\varphi\dd\gamma^{\eps,\lambda}-\int_{\RR^d\times\RR^d}\varphi\dd\gamma^\eps\right|\leq \|\varphi\|_{\infty}\left\|\omega^{\eps,\lambda}-\omega^\eps\right\|_{\ell^1}.
\end{equation*}
By Corollary~\ref{cor:plan-convergence},
\begin{equation*}
  \left\|\omega^{\eps,\lambda}-\omega^\eps\right\|_{\ell^1}\to0\qquad\text{as }\lambda\to\infty.
\end{equation*}
Hence, \eqref{eq:gmm-narrow-convergence} follows.
\end{proof}

The next proposition gives the entropy-selection property for the transport polytope $\Pi(a,b)$. The same selection principle appears in \cite[Proposition~4.1]{cominetti1994asymptotic}. We include a direct proof for completeness.

\begin{prop}\label{prop:entropy-selection}
Let $C\in\RR^{K\times L}$. For each $\eps>0$, let $\omega^\eps$ be the unique minimizer of
\begin{equation}\label{eq:entropy-regularized-linear-problem}
  \min_{\omega\in\Pi(a,b)}
  \left\{
  \langle C,\omega\rangle+\eps E(\omega)
  \right\},
\end{equation}
and define the set of minimizers of the unregularized problem by
\begin{equation}\label{eq:unregularized-optimal-set}
  \mathcal S=\argmin_{\omega\in\Pi(a,b)}\langle C,\omega\rangle.
\end{equation}
Then there exists a unique maximum-entropy optimizer $\omega^*\in\mathcal S$, characterized by
\begin{equation*}
  \omega^*=\argmax_{\omega\in\mathcal S}\bigl(-E(\omega)\bigr).
\end{equation*}
Moreover,
\begin{equation}\label{eq:entropy-selection-limit}
  \omega^\eps\to\omega^*
  \qquad\text{as }\eps\to0.
\end{equation}
\end{prop}

\begin{proof}
The set $\Pi(a,b)$ is compact and convex. Since the linear functional $\omega\mapsto\langle C,\omega\rangle$ is continuous, the set $\mathcal S$ in \eqref{eq:unregularized-optimal-set} is nonempty, compact, and convex. Since $E$ is continuous and strictly convex, it has a unique minimizer $\omega^*$ on $\mathcal S$. Equivalently, $\omega^*$ is the unique maximizer of $-E$ on $\mathcal S$.

Let $\overline\omega$ be an arbitrary cluster point of $\omega^\eps$ as $\eps\to0$. Then there exists a sequence $\eps_r\to0$ such that
\begin{equation*}
  \omega^{\eps_r}\to\overline\omega.
\end{equation*}
For any $\eta\in\Pi(a,b)$, the optimality of $\omega^{\eps_r}$ in \eqref{eq:entropy-regularized-linear-problem} gives
\begin{equation*}
  \langle C,\omega^{\eps_r}\rangle+\eps_r E(\omega^{\eps_r})\leq \langle C,\eta\rangle+\eps_r E(\eta).
\end{equation*}
Since $E$ is bounded on the compact set $\Pi(a,b)$, passing to the limit gives
\begin{equation*}
  \langle C,\overline\omega\rangle
  \leq
  \langle C,\eta\rangle
\end{equation*}
for every $\eta\in\Pi(a,b)$. Therefore, we have $\overline\omega\in\mathcal S$.

Taking $\eta=\omega^*$ in the optimality inequality, we obtain
\begin{equation*}
  \eps_r\left(E(\omega^{\eps_r})-E(\omega^*)\right)\leq\langle C,\omega^*\rangle-\langle C,\omega^{\eps_r}\rangle.
\end{equation*}
Since $\omega^*\in\mathcal S$,
\begin{equation*}
  \langle C,\omega^*\rangle \leq \langle C,\omega^{\eps_r}\rangle.
\end{equation*}
Hence,
\begin{equation*}
  E(\omega^{\eps_r})\leq E(\omega^*).
\end{equation*}
Passing to the limit and using the continuity of $E$, we obtain
\begin{equation}\label{eq:geqE}
  E(\overline\omega)\leq E(\omega^*).
\end{equation}
On the other hand, since $\overline\omega\in\mathcal S$ and $\omega^*$ minimizes $E$ on $\mathcal S$,
\begin{equation}\label{eq:leqE}
  E(\omega^*)\leq E(\overline\omega).
\end{equation}
Therefore, combining inequalities \eqref{eq:geqE} and \eqref{eq:leqE}, we obtain
\begin{equation*}
  E(\overline\omega)=E(\omega^*).
\end{equation*}
Since $\omega^*$ is the unique minimizer of $E$ on $\mathcal S$, we conclude that
\begin{equation*}
  \overline\omega=\omega^*.
\end{equation*}
Thus, $\omega^*$ is the unique cluster point of $\omega^\eps$ as $\eps\to0$. Together with the compactness of $\Pi(a,b)$, this proves \eqref{eq:entropy-selection-limit}.
\end{proof}

\begin{thm}\label{thm:gmm-double-limit}
Let $\gamma^{\eps,\lambda}$ and $\gamma^\eps$ be defined by \eqref{eq:gamma-eps-lambda}. Let $\omega^*$ be the unique maximum-entropy minimizer in the mixture Wasserstein problem \eqref{eq:mw-discrete}, and define
\begin{equation*}
  \gamma^*=\gamma_{\omega^*}.
\end{equation*}
Then for every $\varphi\in C_b(\RR^d\times\RR^d)$,
\begin{equation}\label{eq:gmm-double-limit}
\lim_{\eps\to0}\lim_{\lambda\to\infty}\int_{\RR^d\times\RR^d}\varphi(x,y)\dd\gamma^{\eps,\lambda}(x,y)=\int_{\RR^d\times\RR^d}\varphi(x,y)\dd\gamma^*(x,y).
\end{equation}
Thus,
\begin{equation*}
\lim_{\eps\to0}\lim_{\lambda\to\infty}\gamma^{\eps,\lambda}=\gamma^*
\end{equation*}
in the sense of narrow convergence.
\end{thm}

\begin{proof}
Let $\varphi\in C_b(\RR^d\times\RR^d)$. By Theorem~\ref{thm:gmm-narrow-lambda},
\begin{equation*}
\lim_{\lambda\to\infty}\int_{\RR^d\times\RR^d}\varphi\dd\gamma^{\eps,\lambda}=\int_{\RR^d\times\RR^d}\varphi\dd\gamma^\eps.
\end{equation*}
By \eqref{eq:gmm-induced-plan-test},
\begin{equation*}
  \int_{\RR^d\times\RR^d}\varphi\dd\gamma^\eps-\int_{\RR^d\times\RR^d}\varphi\dd\gamma^*=\sum_{k=1}^{K}\sum_{l=1}^{L}\left(\omega^\eps_{kl}-\omega^*_{kl}\right)\int_{\RR^d}\varphi(x,T_{kl}(x))\dd\mu_k(x).
\end{equation*}
Since $\varphi$ is bounded,
\begin{equation*}
  \left| \int_{\RR^d\times\RR^d}\varphi\dd\gamma^\eps - \int_{\RR^d\times\RR^d}\varphi\dd\gamma^*
  \right| \leq \|\varphi\|_{\infty} \|\omega^\eps-\omega^*\|_{\ell^1}.
\end{equation*}
By Proposition~\ref{prop:entropy-selection} and the equivalence of norms in finite dimensions,
\begin{equation*}
  \|\omega^\eps-\omega^*\|_{\ell^1}\to0
  \qquad\text{as }\eps\to0.
\end{equation*}
Therefore, \eqref{eq:gmm-double-limit} follows.
\end{proof}

\subsection{Entropic partial displacement interpolation}\label{subsec:epot-interpolation}
Let
\begin{equation*}
  \mu^0=\sum_{k=1}^{K}a_k\mu^0_k,\qquad \mu^1=\sum_{l=1}^{L}b_l\mu^1_l
\end{equation*}
be GMMs on the same Euclidean space. Assume that all covariance matrices are positive definite. Let $\omega^{\eps,\lambda}$ be the real--real block of the EPOT minimizer between their component weights. For each $t\in[0,1]$ and each pair $(k,l)$, we define the componentwise McCann displacement interpolation by
\begin{equation}\label{eq:component-interpolation}
  \mu^t_{kl}=\bigl((1-t)\Id+tT_{kl}\bigr)_\#\mu^0_k,
\end{equation}
where $T_{kl}$ is the optimal transport map from $\mu^0_k$ to $\mu^1_l$ defined by \eqref{eq:gaussian-optimal-map}. We define the entropic partial displacement interpolation by
\begin{equation}\label{eq:partial-displacement-interpolation}
  \mu^t_{\eps,\lambda}=\sum_{k=1}^{K}\sum_{l=1}^{L}\omega^{\eps,\lambda}_{kl}\mu^t_{kl}.
\end{equation}
Since each $\mu^t_{kl}$ is a probability measure, the total mass of \eqref{eq:partial-displacement-interpolation} is
\begin{equation*}
|\mu^t_{\eps,\lambda}|=\sum_{k=1}^{K}\sum_{l=1}^{L}\omega^{\eps,\lambda}_{kl}=Z_{\lambda}.
\end{equation*}
Thus, $\mu^t_{\eps,\lambda}$ describes the displacement interpolation of the mass matched by the entropic partial coupling.

\begin{prop}\label{prop:displacement-interpolation-convergence}
For every fixed $\eps>0$ and $t\in[0,1]$, the measures $\mu^t_{\eps,\lambda}$ converge narrowly as $\lambda\to\infty$ to
\begin{equation}\label{eq:balanced-displacement-interpolation}
\mu^t_{\eps}=\sum_{k=1}^{K}\sum_{l=1}^{L}\omega^\eps_{kl}\mu^t_{kl}.
\end{equation}
Moreover, for every $t\in[0,1]$, the measures $\mu^t_{\eps}$ converge narrowly as $\eps\to0$ to
\begin{equation}\label{eq:unregularized-displacement-interpolation}
  \mu^t_*=\sum_{k=1}^{K}\sum_{l=1}^{L}\omega^*_{kl}\mu^t_{kl},
\end{equation}
where $\omega^*$ is the unique maximum-entropy minimizer in the mixture Wasserstein problem \eqref{eq:mw-discrete}.
\end{prop}

\begin{proof}
Let $\varphi\in C_b(\RR^d)$. Since each $\mu^t_{kl}$ is a probability measure,
\begin{equation*}
\left|\int_{\RR^d}\varphi\,\dd\mu^t_{kl}\right|\leq\|\varphi\|_{\infty}.
\end{equation*}
Using \eqref{eq:partial-displacement-interpolation} and \eqref{eq:balanced-displacement-interpolation}, we obtain
\begin{equation*}
  \left|\int_{\RR^d}\varphi\,\dd\mu^t_{\eps,\lambda}-\int_{\RR^d}\varphi\,\dd\mu^t_{\eps}\right| \leq \|\varphi\|_{\infty}\|\omega^{\eps,\lambda}-\omega^\eps\|_{\ell^1}.
\end{equation*}
By Corollary~\ref{cor:plan-convergence}, the right-hand side converges to zero as $\lambda \to \infty$. Therefore, $\mu^t_{\eps,\lambda}$ converges narrowly to $\mu^t_{\eps}$.

Similarly, using \eqref{eq:balanced-displacement-interpolation} and \eqref{eq:unregularized-displacement-interpolation}, we obtain
\begin{equation*}
  \left|\int_{\RR^d}\varphi\,\dd\mu^t_{\eps}-\int_{\RR^d}\varphi\,\dd\mu^t_*\right| \leq \|\varphi\|_{\infty}\|\omega^\eps-\omega^*\|_{\ell^1}.
\end{equation*}
By Proposition~\ref{prop:entropy-selection} and the equivalence of norms in finite dimensions, the right-hand side converges to zero as $\eps \to 0$. Therefore, $\mu^t_{\eps}$ converges narrowly to $\mu^t_*$.
\end{proof}

\section{Partial Gromov--Wasserstein Distance between Gaussian Mixture Models}\label{sec:pmgw}
This section defines the partial mixture Gromov--Wasserstein distance. The construction follows the idea that a GMM can be identified with a discrete probability measure on a finite space of Gaussian components.

\subsection{Partial mixture Gromov--Wasserstein distance}\label{subsec:pmgw-definition}
Let
\begin{equation*}
  \mu=\sum_{i=1}^{K}a_i\mu_i,\qquad \nu=\sum_{j=1}^{L}b_j\nu_j
\end{equation*}
be Gaussian mixtures on $\RR^d$ and $\RR^{d'}$, respectively, where $a_i>0$, $b_j>0$, and
\begin{equation*}
  \sum_{i=1}^{K}a_i=\sum_{j=1}^{L}b_j=1.
\end{equation*}
After merging repeated components, we assume that the components within each mixture are pairwise distinct. We define the finite component spaces
\begin{equation*}
  \mathcal G_\mu=\{\mu_1,\ldots,\mu_K\},\qquad
  \mathcal G_\nu=\{\nu_1,\ldots,\nu_L\},
\end{equation*}
equipped with the restrictions of the $W_2$ distance, and the discrete probability measures
\begin{equation}\label{eq:pmgw-component-measures}
  \alpha_\mu=\sum_{i=1}^{K}a_i\delta_{\mu_i},\qquad
  \alpha_\nu=\sum_{j=1}^{L}b_j\delta_{\nu_j}.
\end{equation}
For $p,q\geq1$ and $\lambda>0$, we define the partial mixture Gromov--Wasserstein distance by
\begin{equation}\label{eq:pmgw-definition}
  \MGW^{\lambda}_{p,q}(\mu,\nu)=\GW^{\lambda}_{p,q}\bigl((\mathcal G_\mu,W_2,\alpha_\mu),(\mathcal G_\nu,W_2,\alpha_\nu)\bigr).
\end{equation}
Similarly, the balanced mixture Gromov--Wasserstein distance is 
\begin{equation}\label{eq:mgw-definition}
  \MGW_{p,q}(\mu,\nu)=\GW_{p,q}\bigl((\mathcal G_\mu,W_2,\alpha_\mu),(\mathcal G_\nu,W_2,\alpha_\nu)\bigr).
\end{equation}

For the weight vectors
\begin{equation*}
  a=(a_1,\ldots,a_K),\qquad
  b=(b_1,\ldots,b_L),
\end{equation*}
we define
\begin{equation*}
  \Pi_{\leq}(a,b)=\left\{\omega\in\RR_{\geq0}^{K\times L}: \omega\ones_L\leq a,\; \omega^\top\ones_K\leq b\right\},
\end{equation*}
where the inequalities are understood componentwise. Equivalently, \eqref{eq:pmgw-definition} can be written as
\begin{equation}\label{eq:pmgw-discrete}
  \left(\MGW^{\lambda}_{p,q}(\mu,\nu)\right)^p=\min_{\omega\in\Pi_{\leq}(a,b)}\left\{\sum_{i,k=1}^{K}\sum_{j,l=1}^{L}\left(\left|W_2(\mu_i,\mu_k)^q-W_2(\nu_j,\nu_l)^q\right|^p-2\lambda\right)\omega_{ij}\omega_{kl}+2\lambda\right\}.
\end{equation}
In the numerical experiments, we use $p=q=2$. In this case, the distortion term in \eqref{eq:pmgw-discrete} is
\begin{equation*}
\left| W_2^2(\mu_i,\mu_k) - W_2^2(\nu_j,\nu_l) \right|^2.
\end{equation*}

\begin{thm}\label{thm:pmgw-properties}
The partial mixture Gromov--Wasserstein distance defined by \eqref{eq:pmgw-definition} has the following properties:
\begin{enumerate}[label=\textup{(\roman*)}]
  \item $\MGW^{\lambda}_{p,q}$ admits a minimizer.
  \item $\MGW^{\lambda}_{p,q}$ defines a metric on the associated component metric measure spaces modulo strong isomorphism. 
  \item If
  \begin{equation*}
    \lambda\geq \max_{\substack{1\leq i,k\leq K\\1\leq j,l\leq L}}\left|W_2(\mu_i,\mu_k)^q-W_2(\nu_j,\nu_l)^q\right|^p,
  \end{equation*}
  then the partial mixture Gromov--Wasserstein distance \eqref{eq:pmgw-definition} coincides with the balanced mixture Gromov--Wasserstein distance \eqref{eq:mgw-definition}. In particular,
  \begin{equation*}
    \MGW^{\lambda}_{p,q}(\mu,\nu) \to \MGW_{p,q}(\mu,\nu) \qquad\text{as }\lambda \to \infty.
  \end{equation*}
\end{enumerate}
\end{thm}

\begin{proof}
The component spaces $\mathcal G_\mu$ and $\mathcal G_\nu$ are finite and hence compact. By \eqref{eq:pmgw-definition}, the partial mixture Gromov--Wasserstein distance is the general partial Gromov--Wasserstein distance applied to the component metric measure spaces
\begin{equation*}
  (\mathcal G_\mu,W_2,\alpha_\mu), \qquad (\mathcal G_\nu,W_2,\alpha_\nu).
\end{equation*}
Therefore, the existence of a minimizer, the metric property modulo strong isomorphism, and the large-penalty statement follow directly from Theorem~\ref{thm:gw-lambda}. 
\end{proof}

\section{Barycentric Projection Maps for Gaussian Mixture Models}\label{sec:barycentric}
Barycentric projection maps convert component couplings into pointwise maps. This section defines the barycentric projection maps associated with the entropic partial Gaussian mixture optimal transport problem \eqref{eq:gmm-epot-value} and the partial mixture Gromov--Wasserstein problem \eqref{eq:pmgw-discrete}. These maps are used in the numerical experiments.

\subsection{Barycentric projection map for the entropic partial Gaussian mixture optimal transport problem}\label{subsec:epot-barycentric}
Let
\begin{equation*}
  \mu=\sum_{k=1}^{K}a_k\mu_k,\qquad \nu=\sum_{l=1}^{L}b_l\nu_l
\end{equation*}
be GMMs on the same Euclidean space. We first recall the construction of a pointwise assignment from a Gaussian mixture transport plan. Let $\omega\in\Pi(a,b)$, and let $\gamma_\omega$ be the induced transport plan defined by \eqref{eq:gmm-induced-plan}. 
Following Delon and Desolneux \cite{delon-2020}, we define the barycentric projection map associated with $\gamma_\omega$ by
\begin{equation*}
  T_b^\omega(x)=\mathbb E_{\gamma_\omega}[Y\mid X=x].
\end{equation*}
Using the formal representation \eqref{eq:gmm-induced-plan-formal} and the identity
\begin{equation*}
  \sum_{l=1}^{L}\omega_{kl}=a_k,
\end{equation*}
we obtain
\begin{equation*}
T_b^\omega(x)=\frac{\displaystyle\sum_{k=1}^{K}\sum_{l=1}^{L}\omega_{kl}p_{\mu_k}(x)T_{kl}(x)}{\displaystyle\sum_{k=1}^{K}a_kp_{\mu_k}(x)}.
\end{equation*}
When $\omega$ is a minimizer of the mixture Wasserstein problem \eqref{eq:mw-discrete}, this is the mean assignment introduced by Delon and Desolneux \cite{delon-2020}.

Let $\omega^\eps$ be the unique minimizer of the balanced entropic optimal transport problem with the Gaussian component cost matrix $C_G$. Following the entropic Gaussian mixture optimal transport construction in \cite{PMID:39710672}, we define
\begin{equation}\label{eq:balanced-entropic-barycentric-map}
T_b^\eps(x)=\frac{\displaystyle\sum_{k=1}^{K}\sum_{l=1}^{L}\omega^\eps_{kl}p_{\mu_k}(x)T_{kl}(x)}{\displaystyle\sum_{k=1}^{K}a_kp_{\mu_k}(x)}.
\end{equation}

We now extend this construction to the entropic partial coupling. Let $\omega^{\eps,\lambda}$ be the real--real block of the EPOT minimizer associated with \eqref{eq:gmm-epot-value}, and let
\begin{equation*}
  \gamma^{\eps,\lambda} = \gamma_{\omega^{\eps,\lambda}}
\end{equation*}
be the induced partial transport plan. We define the entropic partial barycentric projection map by
\begin{equation}\label{eq:epot-barycentric-map}
  T_b^{\eps,\lambda}(x)=\frac{\displaystyle\sum_{k=1}^{K}\sum_{l=1}^{L}\omega^{\eps,\lambda}_{kl}p_{\mu_k}(x)T_{kl}(x)}{\displaystyle\sum_{k=1}^{K}\sum_{l=1}^{L}\omega^{\eps,\lambda}_{kl}p_{\mu_k}(x)}.
\end{equation}
Since the EPOT minimizer is strictly positive and the Gaussian densities are positive, the denominator in \eqref{eq:epot-barycentric-map} is positive for every $x\in\RR^d$.

The next result shows that the entropic partial barycentric projection map converges to the balanced entropic barycentric projection map \eqref{eq:balanced-entropic-barycentric-map} in the large-penalty limit.

\begin{thm}\label{thm:epot-barycentric-convergence}
For every fixed $\eps>0$,
\begin{equation}\label{eq:epot-barycentric-L2-convergence}
  \lim_{\lambda\to\infty}\int_{\RR^d}\left|T_b^{\eps,\lambda}(x)-T_b^\eps(x)\right|^2\dd\mu(x)=0.
\end{equation}
\end{thm}

\begin{proof}
By Corollary~\ref{cor:plan-convergence},
\begin{equation}\label{eq:barycentric-coupling-convergence}
  \left\|
  \omega^{\eps,\lambda}-\omega^\eps
  \right\|_{\ell^1}
  \to0
\end{equation}
as $\lambda\to\infty$. Since $\omega^\eps\in\Pi(a,b)$, for every $k$,
\begin{equation*}
  \left| \sum_{l=1}^{L}\omega^{\eps,\lambda}_{kl}-a_k \right| \leq \left\| \omega^{\eps,\lambda}-\omega^\eps \right\|_{\ell^1}.
\end{equation*}
Therefore, for all sufficiently large $\lambda$,
\begin{equation*}
  \sum_{l=1}^{L}\omega^{\eps,\lambda}_{kl}
  \geq
  \frac{a_k}{2}
\end{equation*}
for every $k$. It follows that
\begin{equation}\label{eq:partial-barycentric-denominator-bound}
  \sum_{k=1}^{K}\sum_{l=1}^{L}
  \omega^{\eps,\lambda}_{kl}p_{\mu_k}(x) \geq \frac{1}{2} \sum_{k=1}^{K}a_kp_{\mu_k}(x).
\end{equation}
By the definition of $T_b^\eps$,
\begin{equation*}
  \sum_{k=1}^{K}\sum_{l=1}^{L}\omega^\eps_{kl}p_{\mu_k}(x)\left(T_{kl}(x)-T_b^\eps(x)\right)=0.
\end{equation*}
Hence,
\begin{equation*}
  T_b^{\eps,\lambda}(x)-T_b^\eps(x)=\frac{\displaystyle\sum_{k=1}^{K}\sum_{l=1}^{L}\left(\omega^{\eps,\lambda}_{kl}-\omega^\eps_{kl}\right)p_{\mu_k}(x)\left(T_{kl}(x)-T_b^\eps(x)\right)}{\displaystyle\sum_{k=1}^{K}\sum_{l=1}^{L}\omega^{\eps,\lambda}_{kl}p_{\mu_k}(x)}.
\end{equation*}
Since
\begin{equation*}
  \sum_{i=1}^{K}a_ip_{\mu_i}(x)
  \geq
  a_kp_{\mu_k}(x),
\end{equation*}
the estimate \eqref{eq:partial-barycentric-denominator-bound} gives
\begin{equation*}
  \left|T_b^{\eps,\lambda}(x)-T_b^\eps(x)\right|\leq \frac{2}{\displaystyle\min_{1\leq i\leq K}a_i}\left\|\omega^{\eps,\lambda}-\omega^\eps\right\|_{\ell^1}\max_{\substack{1\leq k\leq K\\1\leq l\leq L}}\left|T_{kl}(x)-T_b^\eps(x)\right|.
\end{equation*}
Therefore,
\begin{align*}
  &\int_{\RR^d}\left|T_b^{\eps,\lambda}(x)-T_b^\eps(x)\right|^2\dd\mu(x) \\
  &\leq \frac{4}{\left(\displaystyle\min_{1\leq i\leq K}a_i\right)^2}\left\|\omega^{\eps,\lambda}-\omega^\eps\right\|_{\ell^1}^2\int_{\RR^d}\max_{\substack{1\leq k\leq K\\1\leq l\leq L}}\left|T_{kl}(x)-T_b^\eps(x)\right|^2\dd\mu(x).
\end{align*}
Each $T_{kl}$ is affine. Hence, there exists a constant $C_0>0$ such that
\begin{equation*}
  \max_{\substack{1\leq k\leq K\\1\leq l\leq L}}
  |T_{kl}(x)| \leq C_0(1+|x|).
\end{equation*}
Moreover, $T_b^\eps(x)$ is a convex combination of the finitely many values $T_{kl}(x)$. Thus, for some constant $C>0$,
\begin{equation*}
  \max_{\substack{1\leq k\leq K\\1\leq l\leq L}} \left|T_{kl}(x)-T_b^\eps(x)\right|^2 \leq C(1+|x|^2).
\end{equation*}
Since $\mu$ has a finite second moment, the integral on the right-hand side is finite. The convergence \eqref{eq:epot-barycentric-L2-convergence} now follows from \eqref{eq:barycentric-coupling-convergence}.
\end{proof}

\subsection{Barycentric projection map for the partial mixture Gromov--Wasserstein problem}\label{subsec:pmgw-barycentric}
Let
\begin{equation*}
  \mu=\sum_{k=1}^{K}a_k\mu_k,\qquad
  \nu=\sum_{l=1}^{L}b_l\nu_l
\end{equation*}
be GMMs on $\RR^d$ and $\RR^{d'}$, respectively, where $d\geq d'$. Assume that all covariance matrices are positive definite. Let $\omega^\lambda$ be a minimizer of the partial mixture Gromov--Wasserstein problem \eqref{eq:pmgw-discrete} with $p=q=2$.

A mixture Gromov--Wasserstein coupling determines correspondences between Gaussian components but does not directly determine a transport plan between the ambient spaces. Following the alignment-based assignment of Salmona, Delon, and Desolneux \cite{delon-2023}, we associate a pointwise map with the partial component coupling $\omega^\lambda$.

Set
\begin{equation*}
  Z_\lambda=\sum_{k=1}^{K}\sum_{l=1}^{L}\omega^\lambda_{kl},
\end{equation*}
and assume that $Z_\lambda>0$. We define the means of the matched source and target masses by
\begin{equation}\label{eq:pmgw-matched-means}
  m_0^\lambda=\frac{1}{Z_\lambda}\sum_{k=1}^{K}\sum_{l=1}^{L}\omega^\lambda_{kl}m_{0,k},\qquad m_1^\lambda=\frac{1}{Z_\lambda}\sum_{k=1}^{K}\sum_{l=1}^{L}\omega^\lambda_{kl}m_{1,l}.
\end{equation}
The corresponding centered components are
\begin{equation*}
  \overline\mu_k^\lambda=(\tau_{-m_0^\lambda})_\#\mu_k,\qquad \overline\nu_l^\lambda=(\tau_{-m_1^\lambda})_\#\nu_l,
\end{equation*}
where $\tau_c(x)=x+c$.

Let
\begin{equation*}
  \St(d,d') = \left\{ P\in\RR^{d\times d'}: P^\top P=I_{d'} \right\}.
\end{equation*}
We choose an alignment matrix
\begin{equation}\label{eq:alignment-problem}
  P^\lambda \in \argmin_{P\in\St(d,d')} \sum_{k=1}^{K}\sum_{l=1}^{L} \omega^\lambda_{kl} W_2^2\left(\overline\mu_k^\lambda, P_\#\overline\nu_l^\lambda\right).
\end{equation}
For each pair $(k,l)$, let $S_{kl}^\lambda$ be the Gaussian optimal transport map from $\overline\mu_k^\lambda$ to $P^\lambda_\#\overline\nu_l^\lambda$. We define
\begin{equation}\label{eq:pmgw-component-map}
  R_{kl}^\lambda(x) = m_1^\lambda + (P^\lambda)^\top S_{kl}^\lambda\left(x-m_0^\lambda\right).
\end{equation}
Since
\begin{equation*}
  (P^\lambda)^\top P^\lambda=I_{d'},
\end{equation*}
we have
\begin{equation*}
  (R_{kl}^\lambda)_\#\mu_k=\nu_l.
\end{equation*}

The coupling $\omega^\lambda$ and the component maps \eqref{eq:pmgw-component-map} induce the partial transport plan
\begin{equation*}
\gamma^{\GW,\lambda}=\sum_{k=1}^{K}\sum_{l=1}^{L}\omega^\lambda_{kl}(\Id,R_{kl}^\lambda)_\#\mu_k.
\end{equation*}
Its marginals satisfy
\begin{equation*}
(\pi_1)_\#\gamma^{\GW,\lambda}=\sum_{k=1}^{K}\left(\sum_{l=1}^{L}\omega^\lambda_{kl}\right)\mu_k\leq\mu,\qquad (\pi_2)_\#\gamma^{\GW,\lambda}=\sum_{l=1}^{L}\left(\sum_{k=1}^{K}\omega^\lambda_{kl}\right)\nu_l\leq\nu.
\end{equation*}

After normalizing $\gamma^{\GW,\lambda}$ by its total mass $Z_\lambda$, we define its barycentric projection map by
\begin{equation}\label{eq:pmgw-barycentric-map}
  T_b^{\GW,\lambda}(x)=\frac{\displaystyle\sum_{k=1}^{K}\sum_{l=1}^{L}\omega^\lambda_{kl}p_{\mu_k}(x)R_{kl}^\lambda(x)}{\displaystyle\sum_{k=1}^{K}\sum_{l=1}^{L}\omega^\lambda_{kl}p_{\mu_k}(x)}.
\end{equation}
Since $Z_\lambda>0$ and the Gaussian densities are positive, the denominator in \eqref{eq:pmgw-barycentric-map} is positive for every $x\in\RR^d$.

If $\omega^\lambda$ is balanced, then $Z_\lambda=1$, and the matched means \eqref{eq:pmgw-matched-means} coincide with the means of the full mixtures. Hence, the construction reduces to the balanced alignment-based assignment of \cite{delon-2023}.

The alignment problem \eqref{eq:alignment-problem} is computed using the numerical procedure of \cite{delon-2023}, with the balanced component coupling and the full mixture means replaced by $\omega^\lambda$ and the matched means \eqref{eq:pmgw-matched-means}, respectively.

\section{Numerical Experiments}\label{sec:numerical}
This section presents numerical experiments for the proposed entropic partial optimal transport and the partial mixture Gromov--Wasserstein distance. 

\subsection{Entropic partial optimal transport for Gaussian mixtures}\label{subsec:numerics-epot}
We first study entropic partial Gaussian mixture optimal transport between two GMMs on $\RR^2$. The source mixture is denoted by $G_A$ and the target mixture by $G_B$. Their weights, means, and covariance matrices are as follows: 
\begin{equation*}
  \begin{aligned}
  G_A:\quad
  &\alpha_A=(0.5,0.5),\\
  &m_A^{(1)}=(-8.0,-5.0),
  \qquad
  m_A^{(2)}=(1.0,-8.0),\\
  &\Sigma_A^{(1)}
  =
  \begin{pmatrix}
    0.2 & 0.04\\
    0.04 & 0.2
  \end{pmatrix},
  \qquad
  \Sigma_A^{(2)}
  =
  \begin{pmatrix}
    0.2 & -0.02\\
    -0.02 & 0.2
  \end{pmatrix},\\[0.4em]
  G_B:\quad
  &\alpha_B=(0.45,0.45,0.1),\\
  &m_B^{(1)}=(-3.5,3.5),
  \qquad
  m_B^{(2)}=(3.5,3.5),
  \qquad
  m_B^{(3)}=(9.0,7.5),\\
  &\Sigma_B^{(1)}
  =
  \begin{pmatrix}
    0.2 & -0.02\\
    -0.02 & 0.2
  \end{pmatrix},
  \qquad
  \Sigma_B^{(2)}
  =
  \begin{pmatrix}
    0.2 & 0.04\\
    0.04 & 0.2
  \end{pmatrix},
  \qquad
  \Sigma_B^{(3)}
  =
  \begin{pmatrix}
    0.2 & 0\\
    0 & 0.2
  \end{pmatrix}.
  \end{aligned}
\end{equation*}
The third component of $G_B$ is spatially separated from the other components and is treated as an outlying target component. For sufficiently small values of the penalty parameter $\lambda>0$, matching this component is more costly than leaving part of the mass unmatched.

The component cost matrix $C_G$ is computed using the Gaussian squared Wasserstein distance \eqref{eq:intro-gaussian-w2}. In the numerical experiments, we use the normalized cost matrix
\begin{equation}\label{eq:numerical-normalized-cost}
  \widehat C_G = \frac{C_G}{\displaystyle \max_{\substack{1\leq k\leq K\\1\leq l\leq L}}(C_G)_{kl}}.
\end{equation}
Thus, all reported values of the penalty parameter $\lambda \in (0, 0.5]$ and the regularization parameter $\eps>0$ are relative to the scale of the normalized component costs \eqref{eq:numerical-normalized-cost}.

\subsubsection{Single value of \texorpdfstring{$\lambda$}{lambda}}
We first set $\lambda=0.3$ and $\eps=0.01$. Figure~\ref{fig:lambda-fixed} shows the real--real block $\omega^{\eps,\lambda}$ of the optimal coupling in \eqref{eq:gmm-epot-value}. The red and blue density plots show the source and target GMMs, respectively. The width and color of each line indicate the corresponding coupling weight $\omega^{\eps,\lambda}_{kl}$.

\begin{figure}[htbp]
  \centering 
  \includegraphics[width=0.5\textwidth]{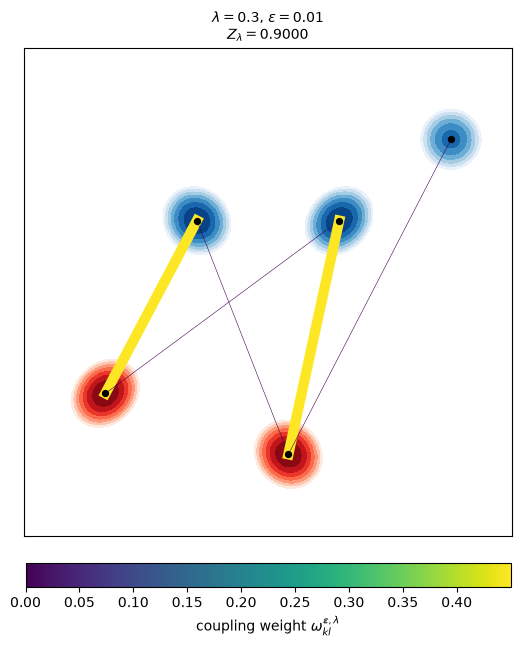}
  \caption{Entropic partial mixture coupling $\omega^{\eps,\lambda}$ for $\lambda=0.3$ and $\eps=0.01$.}
  \label{fig:lambda-fixed}
\end{figure}

The coupling exhibits the expected selective behavior. Most of the transported mass is assigned to the two target components that are geometrically compatible with the source components, while the isolated target component receives only negligible mass.

\subsubsection{Effect of the penalty parameter \texorpdfstring{$\lambda$}{lambda}}
We next fix $\eps=0.01$ and consider
\begin{equation*}
  \lambda\in\{0,0.125,0.25,0.375,0.5\}.
\end{equation*}
We include $\lambda=0$ as a numerical reference, although the theoretical results are stated for $\lambda > 0$, since the finite-dimensional problem \eqref{eq:gmm-epot-value} remains well defined at $\lambda=0$. Figure~\ref{fig:lambda-sweep} shows the resulting real--real coupling blocks.

\begin{figure}[htbp]
  \centering
  \includegraphics[width=0.95\textwidth]{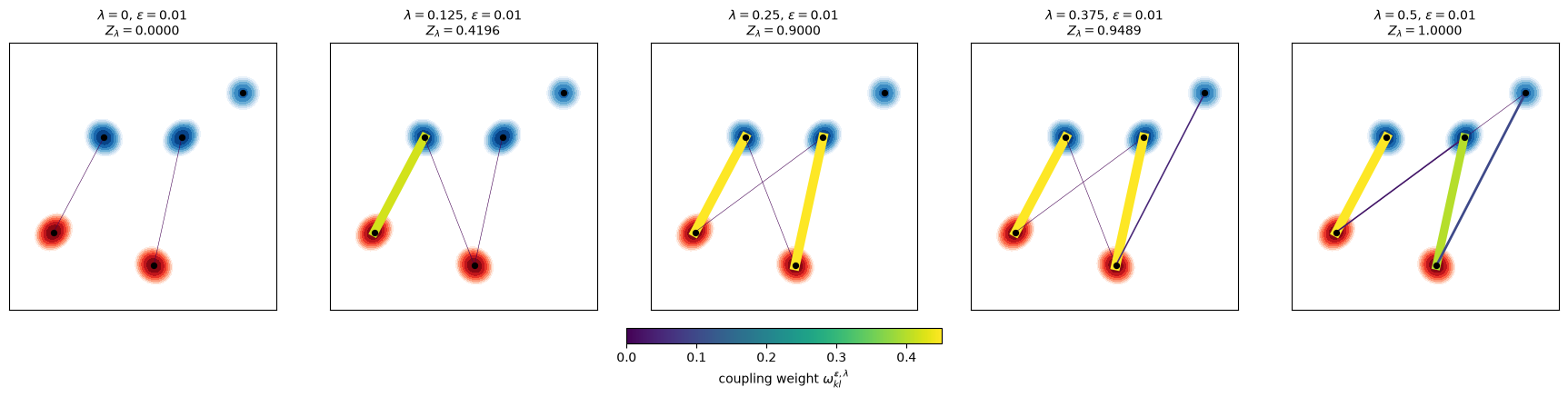}
  \caption{Entropic partial mixture couplings for varying $\lambda$ with fixed $\eps=0.01$. As $\lambda$ increases, a larger amount of mass is matched between the source and target GMMs.}
  \label{fig:lambda-sweep}
\end{figure}

For small values of $\lambda$, only the most compatible component pairs receive substantial transported mass. As $\lambda$ increases, leaving mass unmatched becomes more expensive, and the total matched mass increases in this experiment. This behavior is consistent with Proposition~\ref{prop:mass-convergence} and the convergence of the real--real coupling block established in Corollary~\ref{cor:plan-convergence}.

\subsubsection{Joint effects of the penalty and regularization parameters}
We vary both the penalty parameter $\lambda$ and the entropic regularization parameter $\eps$ in \eqref{eq:gmm-epot-value}. We consider
\begin{equation*}
  \lambda\in\{0.125,0.1875,0.25,0.3125,0.375,0.4375,0.5\},
  \qquad
  \eps\in\{0.01,0.11,0.21\}.
\end{equation*}
Figure~\ref{fig:lambda-epsilon-sweep} shows the resulting real--real coupling blocks.

\begin{figure}[htbp]
  \centering
  \includegraphics[width=\textwidth]{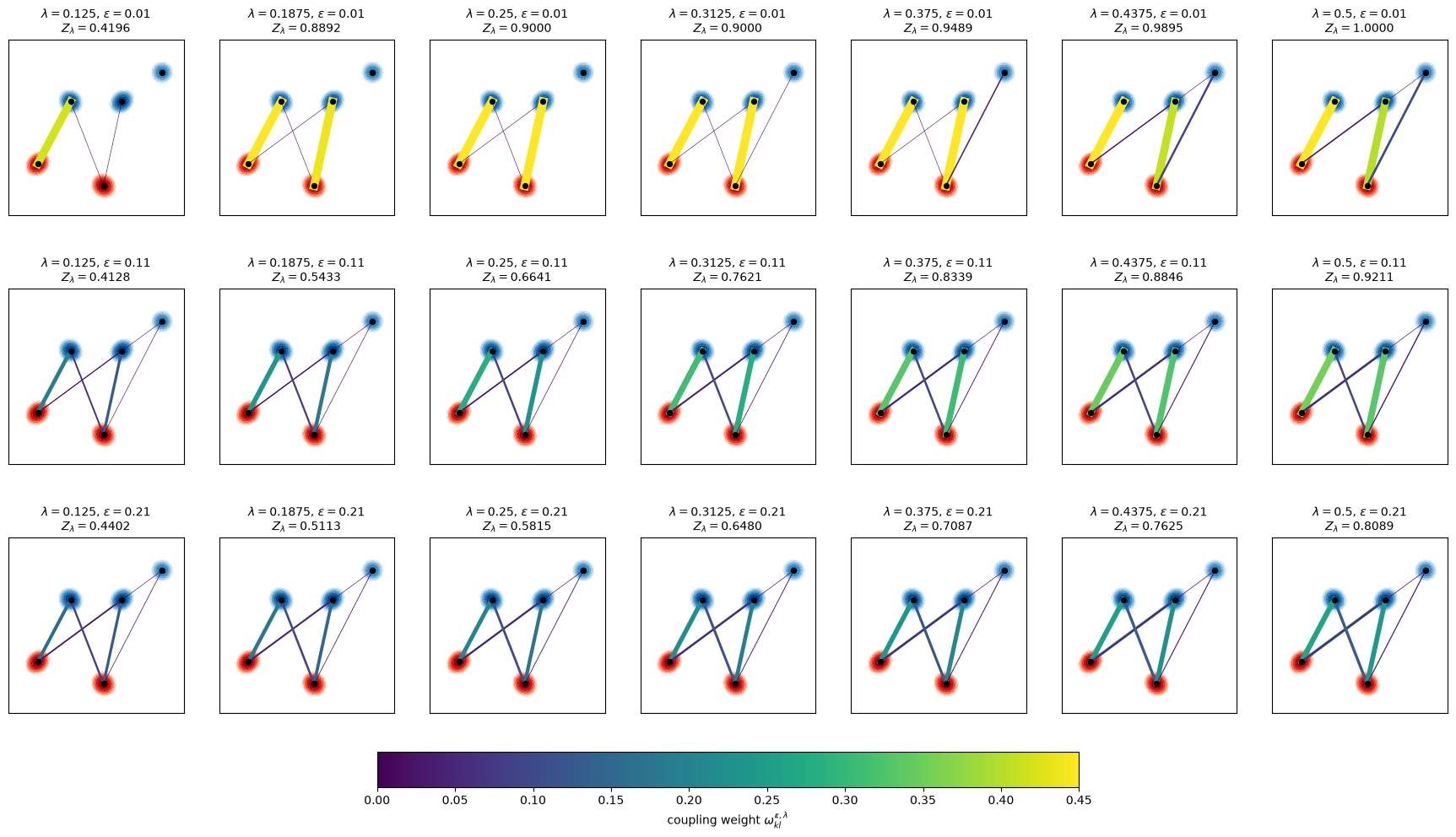}
  \caption{Joint effects of $\lambda$ and $\eps$ on the real--real coupling block.}
  \label{fig:lambda-epsilon-sweep}
\end{figure}

At fixed $\eps$, increasing $\lambda$ increases the amount of matched mass, while at fixed $\lambda$, increasing $\eps$ makes the coupling more diffuse across component pairs.

\subsubsection{Entropic partial displacement interpolation}
We next visualize the entropic partial displacement interpolation $\mu_{\eps,\lambda}^t$ defined by \eqref{eq:partial-displacement-interpolation}. For each parameter pair, the Gaussian interpolation for each component pair is computed according to \eqref{eq:component-interpolation} and weighted by the corresponding coupling weight $\omega^{\eps,\lambda}_{kl}$. Figure~\ref{fig:gmm-flow-grid} shows three representative parameter pairs.

\begin{figure}[htbp]
  \centering
  \includegraphics[width=\textwidth]{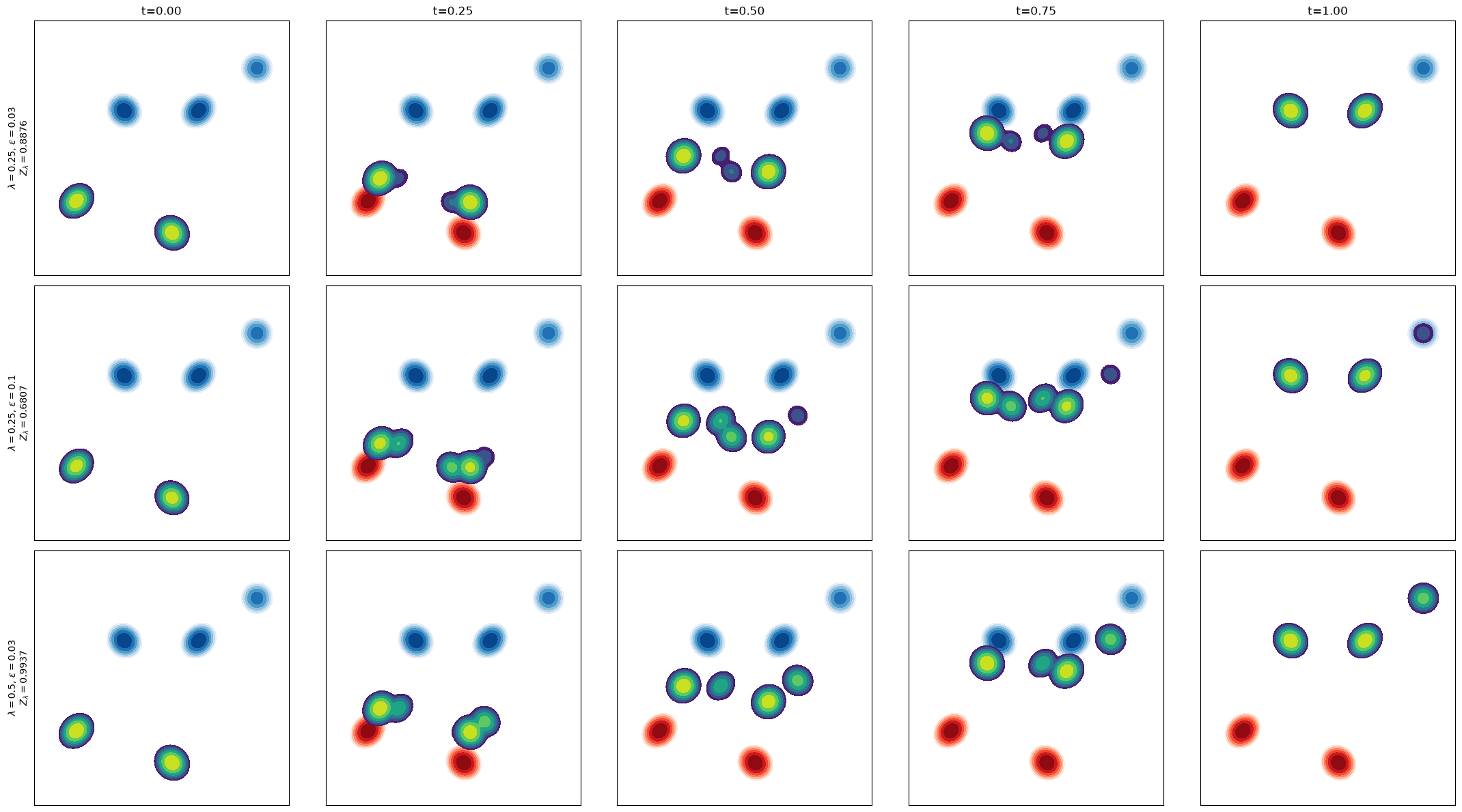}
  \caption{Entropic partial displacement interpolations for $(\eps,\lambda)=(0.03,0.25)$, $(0.10,0.25)$, and $(0.03,0.50)$, shown in the first, second, and third rows, respectively.}
  \label{fig:gmm-flow-grid}
\end{figure}

\subsubsection{Interpolation induced by the entropic partial barycentric projection map}
We visualize the entropic partial barycentric projection map $T_b^{\eps,\lambda}$ defined by \eqref{eq:epot-barycentric-map}. We draw $2000$ samples from each of $G_A$ and $G_B$. For each source sample $x$, we consider the pointwise interpolation
\begin{equation}\label{eq:epot-barycentric-interpolation}
  x_t^{\eps,\lambda}
  =
  (1-t)x+tT_b^{\eps,\lambda}(x),
  \qquad
  t\in[0,1].
\end{equation}
We compare $\lambda=0.25$ and $\lambda=0.5$ with fixed $\eps=0.03$.

\begin{figure}[htbp]
  \centering

  \includegraphics[width=\textwidth]{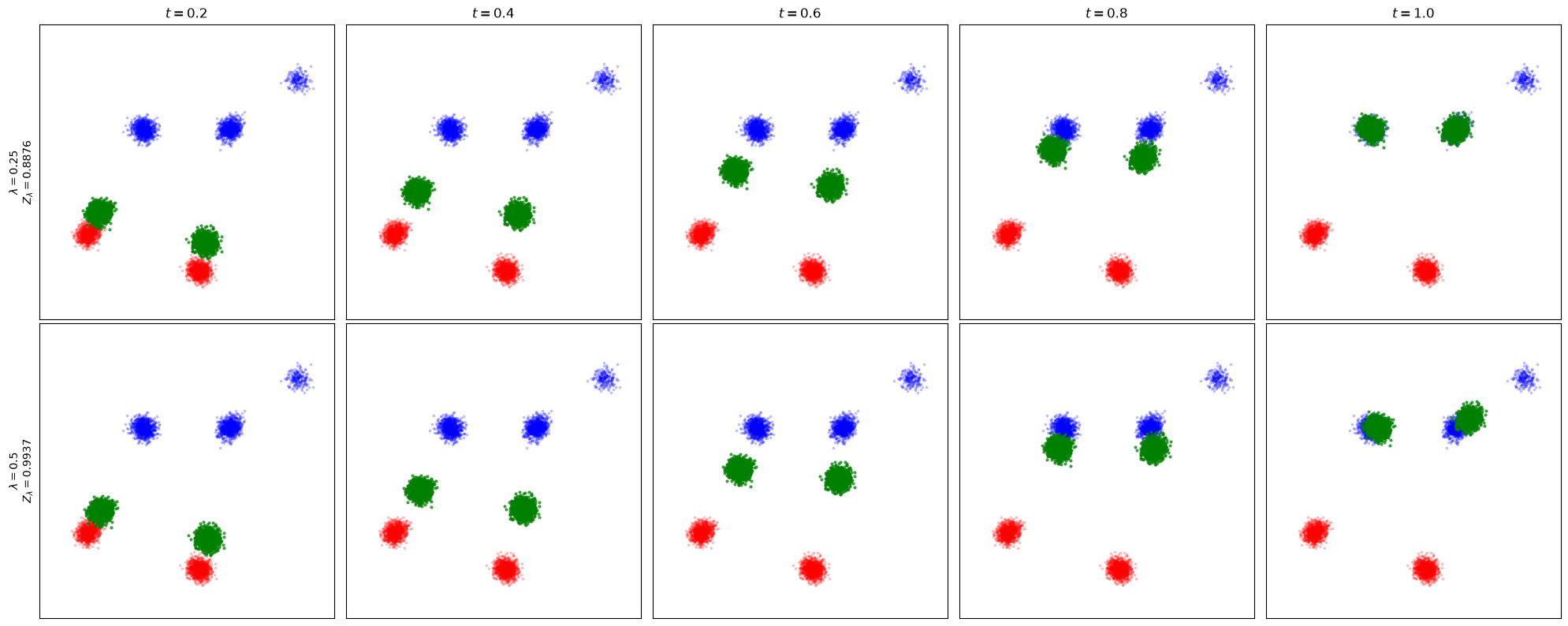}

  \caption{Pointwise interpolations induced by the entropic partial barycentric projection map with fixed $\eps=0.03$. The red and blue points are samples from the source and target GMMs, respectively, and the green points are the interpolated source samples defined by \eqref{eq:epot-barycentric-interpolation}.}
  \label{fig:entropic-partial-barycentric-interpolation}
\end{figure}

For $\lambda=0.25$, the interpolation is primarily governed by the two geometrically compatible target components, while the isolated target component has little influence. For $\lambda=0.5$, the pointwise interpolation is closer to that induced by the balanced entropic barycentric projection map $T_b^\eps$, in qualitative agreement with Theorem~\ref{thm:epot-barycentric-convergence}.

\subsection{Point-cloud matching using partial mixture Gromov--Wasserstein distance}\label{subsec:numerics-pmgw}
We finally compare balanced and partial GW-type methods on two synthetic
point clouds in $\RR^3$. For the partial methods, we set $\lambda=0.01$. The source point cloud $P_1 \subset \RR^3$ consists of $300$ points sampled from a GMM with six equally weighted components whose means are arranged on a circle of radius $4$ in the plane $z=0$. The target point cloud $P_2\subset\RR^3$ consists of $300$ points sampled from a corresponding six-component GMM placed at a different height, together with $50$ additional noise points. Thus, $P_2$ contains a structured part corresponding to $P_1$ and an outlying part that need not be matched.

We compare four methods: balanced GW \eqref{eq:gw-renamed} and partial GW \eqref{eq:pgw-renamed}, applied directly to the empirical point-cloud measures, and balanced mixture GW \eqref{eq:mgw-definition} and partial mixture GW \eqref{eq:pmgw-discrete}, applied to fitted GMMs. All computations use $p=q=2$. For a fair comparison, all cost matrices are normalized before solving the corresponding optimal transport problems. The point-level methods compute couplings directly between individual points. For the mixture methods, the fitted GMMs are identified with the component metric measure spaces defined by \eqref{eq:pmgw-component-measures}, and the component couplings are computed from the corresponding pairwise Gaussian $W_2$ distance matrices. Following \cite{delon-2023}, the alignment matrix is computed by projected gradient descent for \eqref{eq:alignment-problem}. In the partial case, the matched means \eqref{eq:pmgw-matched-means} are used in the alignment construction. 

For a point-level coupling $\gamma=(\gamma_{ij})$, we use the discrete barycentric projection map
\begin{equation*}
  T_b(x_i) = \frac{\displaystyle\sum_j\gamma_{ij}y_j}{\displaystyle\sum_j\gamma_{ij}},
\end{equation*}
whenever $\sum_j\gamma_{ij}>0$. Source points with zero matched
mass are left unassigned. For all four methods, each mapped source
point is then assigned to its nearest point in $P_2$.

Figure~\ref{fig:gw-comparison} compares the nearest-point correspondences induced by the barycentric projection maps. Since balanced GW and balanced mixture GW must match all mass, the outliers in $P_2$ influence the resulting transport plans and produce many inappropriate correspondences. In contrast, the partial methods can leave the outlying part unmatched. 

To further compare the induced pointwise maps, Figure~\ref{fig:gw-bary-comparison} shows the barycentric projection maps before the nearest-point assignment. Although barycentric projection maps generally do not push the source measure forward exactly to the target measure, they are deterministic pointwise maps obtained by taking the barycenters of the conditional target distributions encoded by the underlying transport plans. Even before the nearest-point post-processing, the partial mixture GW barycentric projection map \eqref{eq:pmgw-barycentric-map} shown in Figure~\ref{fig:gw-bary-comparison}(b) preserves the common geometric structure more faithfully than the balanced mixture GW barycentric projection map shown in Figure~\ref{fig:gw-bary-comparison}(a). This indicates that the robustness to outliers is already reflected in the partial component coupling and the induced barycentric projection map.

\begin{figure}[htbp]
  \centering

  \includegraphics[width=\textwidth]{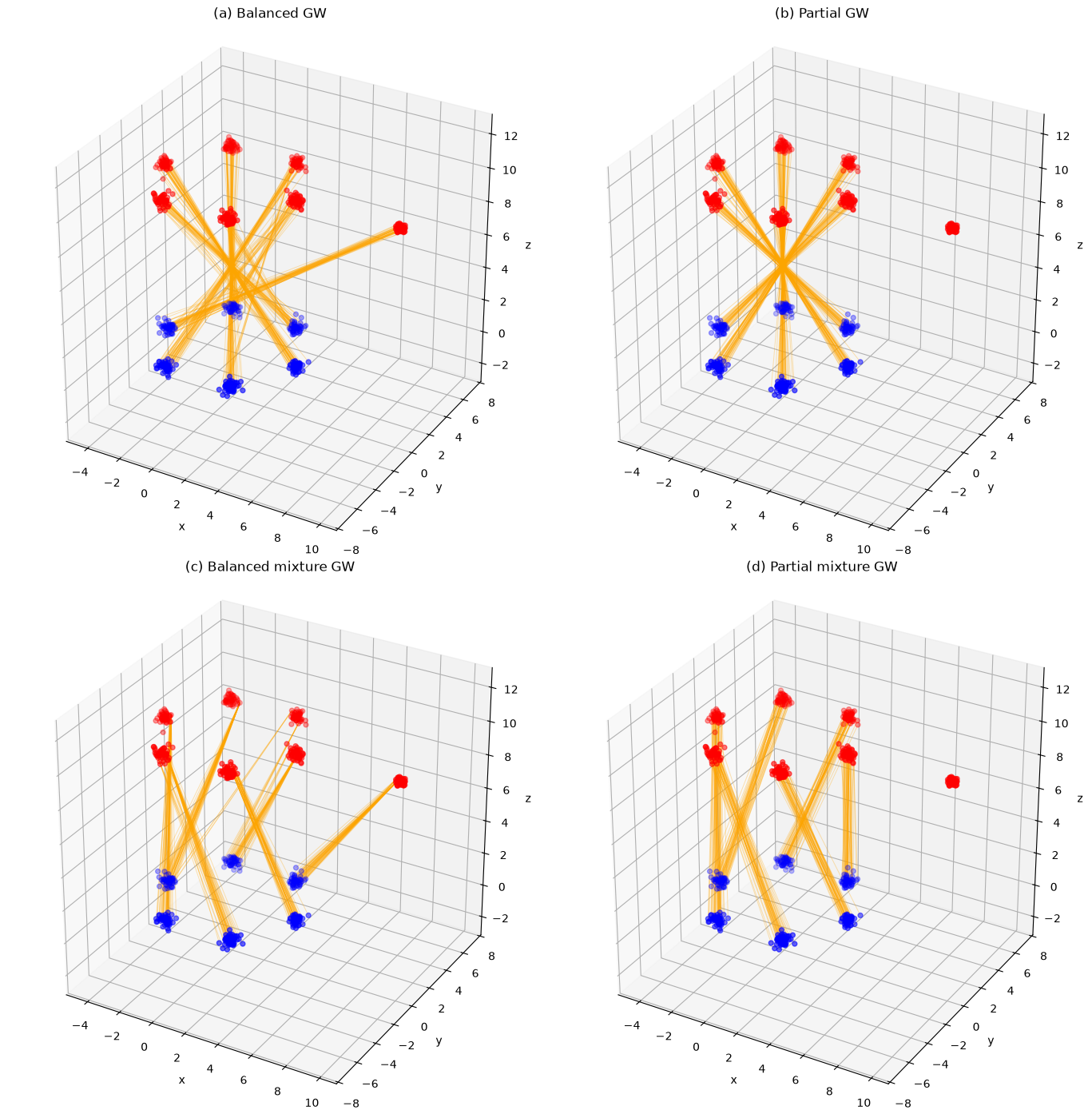}

  \caption{Comparison of balanced and partial GW-type methods on synthetic point clouds with outliers.}
  \label{fig:gw-comparison}
\end{figure}

\begin{figure}[htbp]
  \centering

  \includegraphics[width=\textwidth]{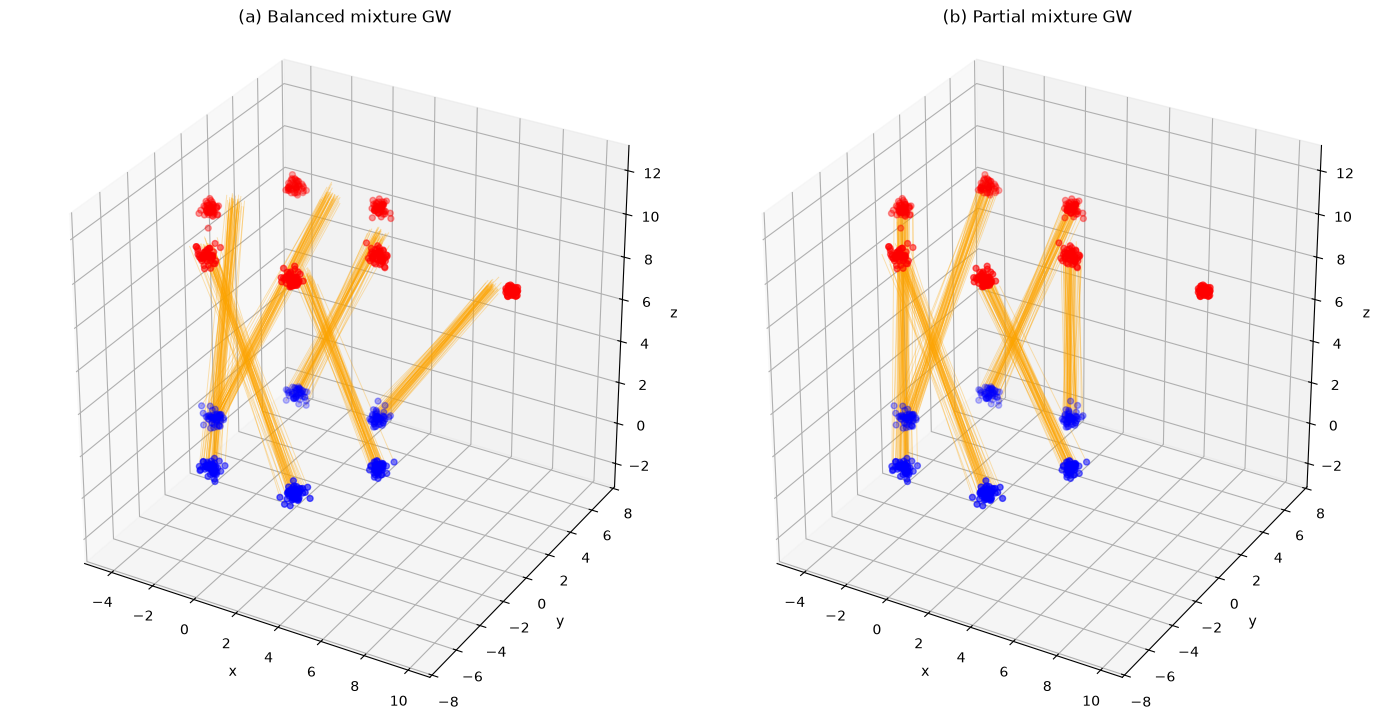}

  \caption{Barycentric projection maps obtained from balanced and partial
mixture GW before nearest-point assignment.}
  \label{fig:gw-bary-comparison}
\end{figure}

The component representation also substantially reduces the computational cost of the coupling problem. Indeed, the point-level coupling between the two empirical measures is replaced by a coupling between a much smaller number of Gaussian components. This reduction considerably decreases the size of the optimization problem and provides a substantial computational advantage over the point-level formulation.

\section{Conclusions}\label{sec:conclusions}
We developed the finite-dimensional entropic partial optimal transport problem \eqref{eq:epot} and applied it to Gaussian mixture models. For fixed $\eps > 0$, we proved existence and uniqueness of the minimizer, obtained the quantitative unmatched-mass estimate \eqref{eq:mass-rate}, and established the value estimate \eqref{eq:value-error-estimate}. In particular, the EPOT value converges to $\OT_{\eps}(a,b)-\eps$ as stated in \eqref{eq:value-convergence}, and the real--real block of the minimizer converges to the balanced entropic coupling by Corollary~\ref{cor:plan-convergence}.

For Gaussian mixtures, the EPOT component coupling induces the continuous partial transport plan \eqref{eq:gmm-induced-plan}. We proved its narrow convergence in \eqref{eq:gmm-narrow-convergence} and characterized the sequential limits $\lambda \to \infty$ and $\eps \to 0$ by \eqref{eq:gmm-double-limit}. The same componentwise construction yields the entropic partial displacement interpolation \eqref{eq:partial-displacement-interpolation}, whose convergence to the balanced and maximum-entropy interpolations is established in Proposition~\ref{prop:displacement-interpolation-convergence}.

We also defined the partial mixture Gromov--Wasserstein distance by \eqref{eq:pmgw-definition}, with the finite formulation \eqref{eq:pmgw-discrete}. It admits minimizers, defines a metric on the associated component metric measure spaces modulo strong isomorphism, and coincides with the balanced mixture GW distance for sufficiently large $\lambda$. We further constructed pointwise assignments through the partial barycentric projection maps \eqref{eq:epot-barycentric-map} and \eqref{eq:pmgw-barycentric-map}. For the entropic partial barycentric projection map, the large-penalty convergence is established in \eqref{eq:epot-barycentric-L2-convergence}.

The numerical experiments demonstrate that the penalty parameter $\lambda>0$ controls the amount of matched mass, whereas the regularization parameter $\eps>0$ controls the concentration of the component coupling. Partial mixture GW preserves the common geometric structure in the presence of outliers, while the component-level GW formulation substantially reduces the computational burden. Future work includes applications to high-dimensional geometric and biological data.

\section*{Acknowledgments}
The authors are grateful to Dr.~Yuki Kimura (deltaex) for his assistance in organizing the source code used in the numerical experiments. The authors used ChatGPT (OpenAI) to improve the wording, grammar, and readability of portions of the manuscript. The authors assume responsibility for all content.

\section*{Code availability statement}
The source code used to generate the data and reproduce the numerical experiments in this paper is publicly available at \url{https://github.com/yachimura-lab/EPGOT-PMGW}.

\section*{Funding}
The first author is partially supported by JST PRESTO (JPMJPR24KD).

\section*{Conflict of interest}
The authors declare that they have no conflicts of interest.

\bibliographystyle{siam}
\bibliography{references}

\bigskip
\noindent
\textsc{Toshiaki Yachimura,\\
Mathematical Science Center for Co-creative Society, Tohoku University,\\
Sendai 980-0845, Japan}\\
\noindent
\emph{Electronic mail address:} toshiaki.yachimura.a4@tohoku.ac.jp

\bigskip

\noindent
\textsc{Xiaocheng Zou,\\
Mathematical Institute, Tohoku University,\\
Sendai 980-8578, Japan}\\
\noindent
\emph{Electronic mail address:} zou.xiaocheng.t3@dc.tohoku.ac.jp

\end{document}